\documentclass[11pt,reqno]{amsart}
\usepackage{amssymb,mathrsfs,color,mathtools,empheq, verbatim, epstopdf}
\usepackage{pinlabel}
\mathtoolsset{showonlyrefs}
\usepackage{hyperref} 

\usepackage{enumerate, bbm}
\usepackage{tensor}

\usepackage{graphicx}
\usepackage{xcolor} 
\usepackage{tensor}
\usepackage{slashed}
\usepackage{cite}

\usepackage[shortlabels]{enumitem}

\usepackage{geometry}

\numberwithin{equation}{section}

\newtheorem{thm}{Theorem}[section]

\newtheorem{lem}[thm]{Lemma}
\newtheorem{prop}[thm]{Proposition}

\theoremstyle{definition} 
\newtheorem{rem}[thm]{Remark}
\newtheorem{defn}[thm]{Definition}

\theoremstyle{remark}

\def\cE {\mathcal{E}}

\def\cW {\mathcal{W}}

\newcommand{\La}{\big\langle}
\newcommand{\Ra}{\big\rangle}

\newcommand{\bs}[1]{\boldsymbol{#1}}
\newcommand{\vd}{\mathrm{d}}

\definecolor{deepgreen}{cmyk}{1,0,1,0.5}

\newcommand{\E}{\mathcal{E}}

\newcommand{\LL}{\mathcal{L}}

\newcommand{\N}{\mathbb{N}}
\newcommand{\R}{\mathbb{R}}

\newcommand{\al}{\alpha}

\newcommand{\de}{\delta}

\newcommand{\lam}{\lambda}
\newcommand{\te}{\theta}

\newcommand{\ta}{\tau}

\newcommand{\io}{\iota}

\newcommand{\De}{\Delta}

\newcommand{\Lam}{\Lambda}

\newcommand{\p}{\partial}

\makeatletter

\newcommand{\Rmnum}[1]{\expandafter\@slowromancap\romannumeral #1@}
\makeatother

\newcommand{\ti}{\widetilde}

\newcommand{\U}{\underline}

\newcommand{\ang}[1]{\left\langle{#1}\right\rangle}
\newcommand{\abs}[1]{\left\lvert{#1}\right\rvert}

\newcommand{\EQ}[1]{\begin{equation}\begin{split} #1 \end{split}\end{equation}}

\newcommand{\Del}[1]{}

\newcommand{\mand}{{\ \ \text{and} \ \  }}

\newcommand{\mfor}{{\ \ \text{for} \ \ }}
\newcommand{\mas}{{\ \ \text{as} \ \ }}

\newcommand{\uD}{\operatorname{D}}

\newcommand{\dr}{\, \mathrm{d}r}
\newcommand{\dt}{\, \mathrm{d}t}

\definecolor{green}{rgb}{0,0.8,0} 

\newcommand{\ud}{\mathrm{d}}

\newcommand{\alp}{\alpha}

\newcommand{\veps}{\varepsilon}

\newcommand{\bfd}{{\bf d}}
\newcommand{\bfe}{{\bf e}}

\newcommand{\bfi}{{\bf i}}

\newcommand{\bfq}{{\bf q}}

\newcommand{\calC}{\mathcal C}

\newcommand{\calE}{\mathcal E}

\newcommand{\calL}{\mathcal L}

\newcommand{\calS}{\mathcal S}
\newcommand{\calT}{\mathcal T}

\newcommand{\calW}{\mathcal W}

\newcommand{\calY}{\mathcal Y}
\newcommand{\calZ}{\mathcal Z}

\begin{document}
	\parindent=0pt
	\title[Soliton resolution for energy-critical NLH in the radial case]{Soliton resolution for the energy-critical nonlinear heat equation in the radial case}
	\author{Shrey Aryan}
\begin{abstract}
We establish the Soliton Resolution Conjecture for the radial critical nonlinear heat equation in dimension $D\geq 3.$ Thus, every finite energy solution continuously resolves into a finite superposition of asymptotically decoupled copies of the ground state and free radiation.
\end{abstract}
\maketitle
\section{Introduction}
\subsection{Setting of the Problem} We study the energy-critical semi-linear heat flow on $\mathbb{R}^{D}$,
\begin{align}\label{eqn:NLH}
\partial_{t} u & =\Delta u+|u|^{\frac{4}{D-2}} u \nonumber \\
u(0, x) & =u_{0}(x),
\end{align}
where $D \geq 3$, under radial symmetry. Let $r=|x|$ denote the radial coordinate, $u=u(t, r) \in \mathbb{R}$, then the above equation reduces to 
\begin{align}\label{eqn:NLH}
\partial_{t} u & =\partial_{r}^{2} u +\frac{D-1}{r} \partial_{r}u+|u|^{\frac{4}{D-2}} u \nonumber \\
u(0, r) & =u_{0}(r).
\end{align}
We restrict ourselves to smooth solutions that remain uniformly bounded in the energy space $\mathcal{E}$, i.e.,
\begin{align*}
\|u\|_{\mathcal{E}}^{2}:=\int_{0}^{\infty}\left[\left(\partial_{r} u(r)\right)^{2}+\frac{(u(r))^{2}}{r^{2}}\right] r^{D-1} \dr < \infty.
\end{align*}
By \cite{brezis-cazenave}, given finite energy data, we let $T_{+}>0$ denote the maximal forward time of existence. Define the nonlinear energy functional associated with \eqref{eqn:NLH} 
\begin{align}\label{eqn:energy}
E(u):=\int_{0}^{\infty} \frac{1}{2}\left(\partial_{r} u(r)\right)^{2} r^{D-1} \mathrm{d} r-\int_{0}^{\infty} \frac{1}{2^*}|u(r)|^{2^*} r^{D-1} \mathrm{d} r,
\end{align}
where $2^*:=\frac{2D}{D-2}$ along with the energy density
\begin{align}\label{eqn:energy density}
    \mathbf{e}(u) := \frac{1}{2} (\partial_r u)^2 - \frac{D-2}{2D} |u|^{\frac{2D}{D-2}}.
\end{align}
Solutions to \eqref{eqn:NLH} are invariant under the scaling
\begin{align*}
u(t, r) \mapsto u_{\lambda}(t, r):=\lambda^{-\frac{D-2}{2}} u\left(t / \lambda^{2}, r / \lambda\right), \quad \lambda>0    
\end{align*}
and \eqref{eqn:NLH} is called energy-critical since $E(u)=E\left(u_{\lambda}\right)$. Testing \eqref{eqn:NLH} by $\partial_{t} u$ and integrating by parts we observe the formal energy identity
\begin{align}
E(u(T))+\int_{0}^{T}\|\mathcal{T}(u(t))\|_{L^{2}}^{2} \mathrm{d} t=E(u(0))    
\end{align}
for each $T>0$, where $\mathcal{T}(u):=\partial_{r}^{2} u+\frac{D-1}{r} \partial_{r} u+|u|^{\frac{4}{D-2}} u$. We define the Aubin--Talenti elliptic solution, $W: \mathbb{R}^{D} \rightarrow \mathbb{R}$, by
\begin{align}\label{eqn:talenti}
W(x):=\left(1+\frac{|x|^{2}}{D(D-2)}\right)^{-\frac{D-2}{2}}    
\end{align}
and recall that by Caffarelli-Gidas-Spruck \cite{CGS}, all entire positive solutions to the stationary equation 
\begin{align}
-\Delta W(x)=|W(x)|^{\frac{4}{D-2}} W(x), \quad x \in \mathbb{R}^D 
\end{align}
are given by Aubin--Talenti bubbles up to sign, scaling, and translation. Since the elliptic solutions $W$ are radially symmetric, we will often denote them $W(x)=W(r)$ with $r=|x|.$ For each $\lambda>0$, we write $W_{\lambda}(r):=\lambda^{-\frac{D-2}{2}} W(r / \lambda)$. 
\subsection{Statement of the Main result}
\begin{thm}[Soliton Resolution]\label{thm:main}
Let $D \geq 3$ and let ${u}(t)$ be a finite energy solution to \eqref{eqn:NLH} with initial data ${u}(0)={u}_0 \in \mathcal{E}$, defined on its maximal forward interval of existence $\left[0, T_{+}\right)$. Suppose that,
$$
\limsup _{t \rightarrow T_{+}}\|{u}(t)\|_{\mathcal{E}}<\infty .
$$
Then either
\newline
\newline
$\mathrm{(i)}$ $T_{+}=\infty$, there exist a time $T_0>0$, an integer $N \geq 0$, continuous functions $\lambda_1(t), \ldots, \lambda_N(t) \in C^0\left(\left[T_0, \infty\right)\right)$, signs $\iota_1, \ldots, \iota_N \in\{-1,1\}$, and ${g}(t) \in \mathcal{E}$ defined by
$$
{u}(t)=\sum_{j=1}^N \iota_j {W}_{\lambda_j(t)}+{g}(t),
$$
such that
$$
\|{g}(t)\|_{\mathcal{E}}+\sum_{j=1}^N \frac{\lambda_j(t)}{\lambda_{j+1}(t)} \rightarrow 0 \text { as } t \rightarrow \infty,
$$
where above we use the convention that $\lambda_{N+1}(t)=\sqrt{t}$;
\newline 
\newline
$\mathrm{(ii)}$ $T_{+}<\infty$, there exist a time $T_0<T_{+}$, a function ${u}^* \in \mathcal{E}$, an integer $N \geq 1$, continuous functions $\lambda_1(t), \ldots, \lambda_N(t) \in C^0\left(\left[T_0, T_{+}\right)\right)$, signs $\iota_1, \ldots, \iota_N \in\{-1,1\}$, and ${g}(t) \in \mathcal{E}$ defined by
$$
{u}(t)=\sum_{j=1}^N \iota_j {W}_{\lambda_j(t)}+{u}^*+{g}(t),
$$
such that
$$
\|{g}(t)\|_{\mathcal{E}}+\sum_{j=1}^N \frac{\lambda_j(t)}{\lambda_{j+1}(t)} \rightarrow 0 \text { as } t \rightarrow T_{+},
$$
where above we use the convention that $\lambda_{N+1}(t)=\sqrt{T_{+}-t}$.   
\end{thm}
\begin{rem}
In the parabolic setting, similar results have been established for the harmonic map heat flow by Jendrej, Lawrie, and Schlag in \cite{lawrie-harmonic-map,lawrie-harmonic-map-nonradial} following the bubbling theory for harmonic maps pioneered by Struwe \cite{struwe}, which was further developed in \cite{topping, topping-annals, struwe, qing-tian,qing}. In fact, as explained in \cite{kim2025classification}, the connection between the harmonic map heat flow and the nonlinear heat equation can be seen by the following identity
\EQ{
\left(\partial_{r}^2 + \frac{\partial_r}{r} - \frac{k^2}{r^2}\right)\left(r^k u\right) 
= r^k\left(\partial_{r}^2 + \frac{(2k+2)-1}{r}\partial_r\right)u,
}
which relates the Laplacian operator for $k$--equivariant functions on $\R^2$ to the one for radial functions on $\mathbb{R}^{2k+2}.$ Thus, we will often use the same language from the harmonic map heat flow literature. Finally, note that for the nonlinear heat equation, Theorem \ref{thm:main} has been conjectured to be true and, during the preparation of this paper, was also raised as an open question by Kim and Merle in \cite{kim2025classification}.
\end{rem}
\begin{rem}
The Soliton Resolution Conjecture states that the evolution of generic solutions to nonlinear dispersive equations decouples as a sum of modulated traveling waves (or Solitons) and a free radiation term. The conjecture arose in the 1970s following the numerical simulations of Fermi–Pasta–Ulam \cite{pasta}, Zabusky–Kruskal \cite{zabusky}. Following the breakthrough work of Kenig and Merle \cite{kenig-merle}, the Soliton Resolution Conjecture has been established for some classes of dispersive PDEs. In particular, for the wave equation, the Soliton Resolution Conjecture has been proved in the radial case for either odd space dimensions or $D=4,6$ in \cite{Duyckaerts-wave:d=3, Duyckaerts-wave-2, Duyckaerts-wave-5, Duyckaerts-wave-6, collot2022soliton} using the method of energy channels. On the other hand, using virial inequalities, Jia and Kenig in \cite{jia-kenig} proved the sequential soliton resolution in dimension $D = 6$. Recently, Jendrej and Lawrie in \cite{lawrie-wave} established the Soliton Resolution Conjecture for the radial energy critical nonlinear wave equation in all space dimensions $D\geq 4$ using a novel argument based on the analysis of collision intervals. They also established the same result for the equivariant energy critical wave maps equation in \cite{lawrie-wavemaps}. Finally, the Soliton Resolution Conjecture has also been established for the equivariant self-dual Chern--Simons--Schrödinger equation by Kim--Kwon--Oh in \cite{oh-chern-simons}.
\end{rem}
\begin{rem}
The nonlinear heat equation with power nonlinearity 
\begin{align}
\partial_t u &=\Delta u+|u|^{p-1} u, \quad \text {on }  (0, T_+)\times \R^D \\
u(0,x) &=u_0(x),\quad  x\in  \R^D
\end{align}
has been the subject of intensive study beginning with the seminal works of Giga and Kohn \cite{giga-1,giga-2}. For a detailed introduction, see the excellent monograph of Quittner and Souplet \cite{quitter}. Since our main theorem deals with the asymptotic behavior near the blow-up time, we briefly review some results in that direction. When the nonlinearity is energy subcritical, i.e., $1<p<\infty$ when $D=1,2$ and $1<p<\frac{D+2}{D-2}$ when $D\geq 3$, then \cite{giga-1,giga-2} showed that any blow-up solution is of Type I, i.e., there exists a constant $C>0$ such that
\begin{align}
    \|u(t)\|_{L^\infty(\R^D)} \leq C (T_{+}-t)^{-\frac{1}{p-1}}.
\end{align}
Otherwise, the blow-up is of Type II. In the energy critical setting, i.e., $p=\frac{D+2}{D-2},$ Filippas, Herrero and Velázquez \cite{Filippas}, established that the solution exhibits Type I blow-up if the initial data is positive and radially decreasing in dimension $D \geq 3.$ This result was later improved by Matano and Merle \cite{matano2004nonexistence} where they proved, amongst other things, the same result with radial, positive and bounded initial data. Next, assuming that the initial datum is close to the ground state, Collot, Merle, and Raphaël in \cite{collot-merle} established a Trichotomy in dimension $D\geq 7$; the solution either dissipates to zero, or approaches a rescaled Aubin--Talenti solution, or blows up in finite time, and the blow-up is of Type I. Thus, if the initial data is close to a ground state in dimension $D\geq 7$, the solution does not exhibit Type II blow-up. Recently, Wang and Wei in \cite{wang2021refined} established that for positive initial datum in $L^\infty(\R^D)$ and $D\geq 7$, any blow-up solution is of Type I. In contrast to the previous results, Theorem~\ref{thm:main} is concerned with finite energy solutions that either exist globally in time or exhibit Type II blow-up. Examples of such solutions exhibiting non-trivial bubble tree decompositions have been recently constructed for the critical nonlinearity in dimensions $D\geq 7$ by del Pino, Musso, and Wei in \cite{del_Pino_2021}. Furthermore, the recent work of Kim and Merle \cite{kim2025classification} shows that under the assumption of radial symmetry, the above bubble tree constructions are the only possible examples in dimensions $D\geq 7.$ 
\end{rem}
\subsection{Summary of the Proof}Our proof is inspired by the recent breakthrough works of Jendrej--Lawrie, in particular, \cite{lawrie-harmonic-map}, where they established an analogous version of Theorem~\ref{thm:main} for the harmonic map heat flow in the equivariant setting. However, we encounter several difficulties arising from the non-positivity of the nonlinear energy and slow decay of $\Lambda W \footnote{Here $\Lambda= r\partial_r + \frac{D-2}{2}$}$ in lower dimensions $D<6$.  

The first obstacle also arises in the context of the energy-critical wave equation, where it is addressed using arguments relying on the finite speed of propagation. Since this property does not hold for solutions to \eqref{eqn:NLH}, we develop new energy estimates that are almost monotone along the heat flow in regions away from the origin. While the proof of these monotonicity formulas is straightforward, using them to establish Sequential Soliton Resolution is novel and potentially applicable to other problems involving focusing nonlinearities. 

On the other hand, the slow decay of $\Lambda W$ makes the modulation analysis more delicate in dimensions $D \in \{3,4,5\}$. The authors in \cite{lawrie-wave} and \cite{lawrie-wavemaps} circumvented this issue by performing a refined modulation analysis separately in dimensions $D\in \{4,5\}$. Instead of following their approach, we use our modified energy estimates, yielding a unified argument that works in all dimensions $D\geq 3$. 

We now explain the proof strategy in more detail. First, we define a multi-bubble configuration.
\begin{defn}[Multi-bubble configuration]\label{defn:multi-bubble}
Given $M \in\{0,1, \ldots\}, \vec{\iota}=\left(\iota_1, \ldots, \iota_M\right) \in$ $\{-1,1\}^M$ and an increasing sequence $\vec{\lambda}=\left(\lambda_1, \ldots, \lambda_M\right) \in(0, \infty)^M$, a multi-bubble configuration is defined by the formula
$$
\mathcal{W}(\vec{\iota}, \vec{\lambda} ; r):=\sum_{j=1}^M \iota_j W_{\lambda_j}(r) .
$$
When $M=0$, $\mathcal{W}(\vec{\iota}, \vec{\lambda} ; r):=0$.
\end{defn}
With this definition, we define a localized distance function to multi-bubble configurations by 
\begin{align}\label{defn:loc-dist}
{\delta}_R(u):=\inf _{M, \vec{\io}, \vec{\lambda}}\Bigl(\|u-\mathcal{W}(\vec{\iota}, \vec{\lambda})\|_{\mathcal{E}(r \leq R)}^2+\sum_{j=1}^M\left(\frac{\lambda_j}{\lambda_{j+1}}\right)^{\frac{D-2}{2}}\Bigr)^{\frac{1}{2}}    
\end{align}
where the infimum is taken over all $M \in\{0,1,2, \ldots\}$, all vectors $\vec{\iota} \in\{-1,1\}^M, \vec{\lambda} \in$ $(0, \infty)^M$, and we use the convention that the last scale $\lambda_{M+1}=R$. 

The first step in the Jendrej--Lawrie framework is to establish a localized sequential compactness lemma, which essentially states that a sequence of functions with vanishing tension converges (locally in space) to a multi-bubble configuration. Thus, given a sequence of functions $u_n$  with bounded energy,  a sequence $\rho_n \in (0, \infty)$ of scales,  and tension $\calT$ vanishing in $L^2$ relative to the scale $\rho_n$, i.e., $ \lim_{n \to \infty}  \rho_n \| \calT(u_n) \|_{L^2} = 0$, there exists a subsequence of the $u_n$ that converges to a multi-bubble configuration up to large scales relative to $\rho_n$, i.e., $ \lim_{n \to \infty} \bs \de_{R_n \rho_n} ( u_n)  = 0$ for some sequence $R_n \to \infty$. Fortunately, sequential compactness for the nonlinear critical heat equation has been established recently in \cite{lawrie-viet}.
\begin{lem}[Localized sequential bubbling]\label{lem:loc-seq}
Let $u(t)$ be the solution to \eqref{eqn:NLH} with smooth initial data $u(0)=u_0 \in \mathcal{E}$, defined on its maximal interval of existence $\left[0, T_{+}\right)$. Suppose that
$$
\limsup_{t \rightarrow T_{+}}\|u(t)\|_\mathcal{E}<\infty .
$$
Then either
\newline
$\mathrm{(i)}$ $T_{+}=\infty$, there exist a sequence of times $t_n \rightarrow \infty$, and a sequence $R_n \rightarrow \infty$ such that,
$$
\lim _{n \rightarrow \infty} \delta_{R_n \sqrt{t_n}}\left(u\left(t_n\right)\right)=0 .
$$
$\mathrm{(ii)}$ $T_{+}<\infty$, there exist a sequence of times $t_n \rightarrow T_{+}$, and a sequence $R_n \rightarrow \infty$ such that,
$$
\lim _{n \rightarrow \infty} \delta_{R_n \sqrt{T_{+}-t_n}}\left(u\left(t_n\right)\right)=0 .
$$
\end{lem}
The proof of the above result is a consequence of the following lemma, which we recall below since it will be used in the final section of the proof.
\begin{lem}[Compactness Lemma]\label{lem:compactness}
Let $u_n \in \mathcal{E}$ be a sequence with $\lim \sup _{n \rightarrow \infty}\left\|u_n\right\|_{\mathcal{E}}<\infty$. Let $\rho_n \in(0, \infty)$ be a sequence and suppose that
$$
\lim _{n \rightarrow \infty}\left(\rho_n\left\|\mathcal{T}\left(u_n\right)\right\|_{L^2}\right)=0 .
$$
Then, there exists a sequence $R_n \rightarrow \infty$ so that, up to passing to a subsequence of the $u_n$, we have,
$$
\lim _{n \rightarrow \infty} \delta_{R_n \rho_n}\left(u_n\right)=0 .
$$
The subsequence $u_n$ can be chosen so that there are fixed $(M, \vec{\iota}) \in (\mathbb{N} \cup\{0\}) \times\{-1,1\}^M$, a sequence $\vec{\lambda}_n \in(0, \infty)^M$, and $C_0>0$ with
$$
\lim _{n \rightarrow \infty}\Bigl(\|u_n-\mathcal{W}(\vec{\iota}, \vec{\lambda}_n)\|_{\mathcal{E}(r \leq R_n \rho_n)}^2+\sum_{j=1}^{M-1} \left(\frac{\lambda_{n, j}}{\lambda_{n, j+1}}\right)^{\frac{D-2}{2}}\Bigr)=0,
$$
and,
$$
\lambda_{n, M} \leq C_0 \rho_n, \quad \forall n .
$$
\end{lem}
The next step is to use Lemma \ref{lem:loc-seq} to prove the Soliton Resolution Conjecture along a sequence of times. This has been done using energy estimates in the context of the harmonic map heat flow in \cite{lawrie-harmonic-map}. The key idea there was to establish localized energy monotonicity inequalities, which were used to deduce the size of the energy at later times in a suitable region. However, in our case, the localized nonlinear energy $\int_{0}^{\infty} \left(\frac{1}{2}|\nabla u|^2 -\frac{1}{2^*} |u|^{2^*}\right)\phi^2 r^{D-1} \dr$ is not necessarily non-negative, which prevents us from directly adapting the arguments in \cite{lawrie-harmonic-map}.

We overcome this obstacle by working with the localized $\E$-norm (see \eqref{defn:modified e-norm}) and establishing an almost monotonicity formula. In particular, defining the modified energy density
\EQ{
\Tilde{\mathbf{e}}(u):=|\partial_r u|^2 + \frac{u^2}{r^2}
}
we get
\begin{align}\label{ineq:almost-montone}
&\int_0^{\infty} \Tilde{\mathbf{e}}(u(t_2)) \phi^2 r^{D-1}\dr - \int_0^{\infty} \Tilde{\mathbf{e}}(u(t_1))\phi^2 r^{D-1}\dr   \\
&\leq -\int_{t_1}^{t_2}\int_0^{\infty} (\partial_t u)^2 \phi^2 r^{D-1}\dr \dt + 4 \int_{t_1}^{t_2}\int_0^{\infty} |\partial_r u|^2|\partial_r \phi|^2 r^{D-1}\dr \dt  \\
&\quad +2 \left(\int_{t_1}^{t_2}\int_0^{\infty} |u|^{2p} \phi^2 r^{D-1}\dr \dt \right)^{1/2}  \left(\int_{t_1}^{t_2}\int_0^{\infty} (\partial_t u)^2 \phi^2 r^{D-1}\dr \dt\right)^{1/2}  \\
&\quad +2 \left(\int_{t_1}^{t_2} \int_0^{\infty} \frac{|u|^2}{r^4} \phi^2 r^{D-1}\dr \dt\right)^{1/2}  \left(\int_{t_1}^{t_2} \int_0^{\infty}(\partial_t u)^2 \phi^2 r^{D-1}\dr \dt \right)^{1/2}
\end{align}
where $\phi: I \times (0,\infty)\to [0,\infty)$ is a smooth function defined on a time interval $I\subset [0,\infty)$, $0<t_1<t_2$ with $t_1,t_2\in I$ and $\partial_t \phi \leq 0.$ Since the modified energy density $\Tilde{\mathbf{e}}(u)$ is non-negative we can conveniently propagate smallness estimates from some initial time $t_1$ to a later time $t_2$. On the other hand, we pay a price since the nonlinearity appears in the RHS of \eqref{ineq:almost-montone} that requires us to control the $L^{2p}$ norm of the solution $u$, which falls well outside the range of the standard Sobolev embedding $L^{p+1}(\R^D) \hookrightarrow \dot{H}^1 (\R^D)$. 

We resolve this issue by using the radial Sobolev embedding, which gives us pointwise decay of the solution $u$ away from the origin. Thus localizing the modified energy $\Tilde{E}(u)=\int_0^{\infty} \Tilde{\mathbf{e}}(u) r^{D-1}\dr$ away from the origin, we deduce a localized energy monotonicity formula that is remarkably similar to the one derived for the harmonic map heat flow in \cite{lawrie-harmonic-map}. However, this restriction on the localization region introduces more technical difficulties.

When $T_+=\infty$, we follow the arguments in \cite{lawrie-harmonic-map} along with one new idea, which is required to show the positivity of the nonlinear energy away from the origin. To this end, we prove Lemma \ref{lem:localized-coercivity} where we show that the nonlinear energy of $u$ on the tail region $[r_0,\infty)$, for any $r_0>0$, is non-negative provided its $\E$--norm on $[r_0,\infty)$ is small enough. Thus, this lemma allows us to deduce the non-negativity of the nonlinear energy on tails, where we know that the solution $u$ has a small $\E$--norm since most of its energy is captured by the $N$-bubble configuration in some ball centered at the origin.

On the other hand, when $T_+< \infty$, the argument is more involved. This is because the energy of the solution $u$ is now given by the $N$-bubbles and a weak limit of the flow, denoted by $u^*$. Furthermore, we cannot localize the modified energy $\Tilde{E}$ near the origin, and thus, we cannot access the bubbling region directly, forcing us to argue indirectly. This issue does not arise in the harmonic map heat flow case, where, for instance, one can directly show that the energy of the solution $u$ in the parabolic region is asymptotically equal to the sum of the energies of the $N$-bubbles. This fact, in turn, allows the authors in \cite{lawrie-harmonic-map} to deduce that the energy of $u-u^*$ away from the parabolic region is small. Unfortunately, this line of reasoning does not work in our case, and thus, we are forced to argue in the reverse direction, which makes the argument more technical. See Section \ref{sec:finite-time} for more details.

Having established Sequential Soliton Resolution, we conclude that there exists a sequence of times $t_n\to T_+$ such that the distance to the $N$-bubble configuration denoted by $\bfd(t)$ (see Definition~\ref{def:proximity}), satisfies
\EQ{
\lim_{n \to \infty} \bfd(t_n)  = 0. 
}
Theorem~\ref{thm:main} follows from the fact that $\lim_{t \to T_+} \bfd(t)  = 0$. To prove this, we argue by contradiction, i.e., suppose that $\bfd(t)$ does not converge to zero. Then, heuristically, this means that $\bfd(t)$ is large along a certain sequence of times, or, in other words, $u(t)$ deviates from the $N$-bubble configuration. Instead of analyzing this sequence of times, the key insight in \cite{lawrie-wavemaps} is to consider a sequence of time intervals that keeps track of the dynamical history of the $N$-bubble configuration. Thus consider $[a_n,b_n]$, a sequence of time intervals where near the endpoints $a_n$ and $b_n$, $u(t)$ is close to an $N$-bubble configuration while inside $[a_n,b_n]$, $u(t)$ is close to an $(N-K)$-bubbles, where we choose $K$ to be the smallest non-negative integer such that the above properties hold. Observe that intuitively, $K\geq 1$ since otherwise, $u(t)$ will always be close to an $N$-bubble configuration, which would imply that $\lim_{t \to T_+} \bfd(t)  = 0$. 

Next, we can derive differential inequalities for the scales of the bubbles that come into collision by using modulation theory on these collision intervals. Let $\lambda_K$ denote the scale of the $K$-th bubble that loses its shape, i.e., it undergoes a collision on $[a_n,b_n].$ Then we show that for $n$ large enough, there exists $[c_n,d_n]\subset [a_n,b_n]$ such that 
\begin{align}\label{eqn:interval-inequality}
d_n - c_n \gtrsim n^{-1} \lam_K(c_n)^2.
\end{align}
Similar inequalities have been established in the context of nonlinear waves \cite{lawrie-wave,lawrie-wavemaps} and the harmonic map heat flow \cite{lawrie-harmonic-map}. However, the presence of a negative eigenvalue of the linearized operator (see Section \ref{subsec:multi-bubble})  gives rise to an unstable direction, which makes our modulation analysis similar to that for the nonlinear wave equation. By controlling this unstable direction and adapting the arguments in \cite{lawrie-wave}, one can expect to deduce the inequality \eqref{eqn:interval-inequality} when $D\geq 6$. We also expect that by using the refined modulation as in \cite{lawrie-wave}, one can push this down to $D\geq 4.$ However, it is not clear how to make this strategy work when $D=3.$ Thus, our last contribution is to prove Theorem \ref{thm:main} directly when $D\geq 3$ by getting around the technical difficulties arising from controlling the unstable direction in lower dimensions. To this end, we use energy estimates instead of differential inequalities for the unstable direction, thus avoiding any technical restrictions on the dimension, to prove \eqref{eqn:interval-inequality}. The key insight in this step is to realize that while the nonlinear heat equation has infinite speed of propagation, localized energy inequalities can be used to recover integral analogs of finite speed of propagation, similar to the nonlinear wave case. As a consequence, we can show that if a solution to \eqref{eqn:NLH} is close to an $N$-bubble configuration at some time in a region away from the origin (recall that our energy estimates work away from the origin) then this bound propagates for a short time implying that the distance to the $N$-bubble configuration still remains relatively small. This claim suffices to establish \eqref{eqn:interval-inequality} when $D\geq 3$ without having to resort to estimates used in \cite{lawrie-wave}, where the authors proved a similar fact by controlling, among other things, the size of the unstable mode, which in turn required them to restrict the dimensions to $D\geq 6$ or resort to refined modulation analysis when $D\in \{4,5\}$.

Combined with Lemma \ref{lem:loc-seq} this implies that $\inf_{t \in [c_n, d_n]}\lam_K(c_n)^2\| \calT(u(t))  \|_{L^2}^2 \gtrsim 1$. Thus using the fact that $\int_0^{T_+} \| \calT(u(t))  \|_{L^2}^2 \mathrm{d}t<\infty$ and \eqref{eqn:interval-inequality} we get that
\EQ{
C \ge \int_0^{T_+} \| \calT(u(t))  \|_{L^2}^2 \, \ud t  \ge \sum_{n\in \N} \int_{c_n}^{d_n} \| \calT(u(t))  \|_{L^2}^2 \, \ud t  \gtrsim \sum_{n\in \N}  n^{-1}=\infty, 
}
which is a contradiction and therefore $\lim_{t \to T_+} \bfd(t)= 0$.

\subsection{Notation} 
We use the following notations:
\begin{itemize}
    \item Given a function $\phi(r)$ and $\lambda>0$, we denote by $\phi_\lambda(r)=\lambda^{-\frac{D-2}{2}} \phi(r / \lambda)$, the $\dot{H}^1$-invariant rescaling, and by $\phi_{\underline{\lambda}}(r)=\lambda^{-\frac{D}{2}} \phi(r / \lambda)$ the $L^2$-invariant rescaling. Furthermore, we set $\Lambda:=r \partial_r+\frac{D-2}{2}$ and $\underline{\Lambda}:=r \partial_r+\frac{D}{2}$ as the infinitesimal generators of these scaling.
    \item Given two functions $f,g \in L^2((0,\infty),r^{D-1}\dr)$, we define their inner product
    \EQ{
\langle f \mid g \rangle:=\int_0^{\infty} f(r) g(r) r^{D-1} \dr.    
}
\item Given $u\in \cE$, we define the modified energy density $\Tilde{\bfe}$ and the localized $\cE$ norm as follows
\begin{align}\label{defn:modified e-norm}
\Tilde{\mathbf{e}}(u) := (\partial_r u)^2 + \frac{u^2}{r^2},\quad \Tilde{E}(u;r_1,r_2) := \| u\|_{\cE(r_1, r_2)}^2 := \int_{r_1}^{r_2} \tilde{\mathbf{e}}(u) r^{D-1}\vd r.
\end{align}
By convention, $\cE(r_0) := \cE(r_0, \infty)$ for $r_0 > 0$. We similarly define the nonlinear energy density and the localized nonlinear energy as follows
\EQ{
{\mathbf{e}}(u) := \frac{(\partial_r u)^2}{2} -\frac{|u|^{2^*}}{2^*},\quad  E(u;r_1,r_2) :=\int_{r_1}^{r_2} {\mathbf{e}}(u) r^{D-1}\vd r,
}
where $D\geq 3$, and $2^*:=\frac{2D}{D-2}$.
\item All the spatial integrals are with respect to the measure $r^{D-1}\ud r$ while all the integrals in time are with respect to the measure $\ud t.$ We will often omit these measures in integral identities for the sake of presentation.
\item Throughout the paper, the function $\chi\in C^\infty_c([0,\infty))$ denotes a smooth radial cut-off function, supported on $r\leq 2$ and $\chi\equiv 1$ when $r\leq 1$. Furthermore we denote $\chi_R(r):= \chi(r/R).$
\item The inequality $A \lesssim B$ means that $A \leq C B$ and $A \gtrsim B$ means that $A \geq c B$ for some constants $c, C>0$ possibly depending on the number of bubbles $N.$ We write $A \ll B$ if $\lim _{n \rightarrow \infty} A / B=0$.
\end{itemize}
\subsection{Acknowledgments}
The author is grateful to Andrew Lawrie for suggesting this problem and for his constant encouragement and many insightful conversations throughout the development of this work. The author would also like to thank Tobias Colding, Jacek Jendrej, Kihyun Kim, Yvan Martel, and Michael Struwe for their interest and helpful discussions. Finally, the author is indebted to the referee for their careful reading and constructive comments, which substantially improved the manuscript.
\section{Preliminaries}
\subsection{Local Cauchy Theory}
We first recall the local well-posedness theory for the heat equation in the energy space. See for instance, Theorem 1 in \cite{brezis-cazenave}, Theorem 2.1 in \cite{roxanas}, or Proposition 2.1 in \cite{collot-merle,ikeda2025globaldynamicsenergycriticalnonlinear}.
\begin{lem}[Local well-posedness]\label{lem:lwp}
Assume $D\geq 3$. Given any $u_0 \in \dot{H}^1(\R^D)$, there exist a time $T=T\left(u_0\right) \in (0,\infty]$ and a unique function $u \in C([0, T], \dot{H}^1(\R^D))$ with $u(0)=u_0$ such that $u$ is a classical solution of \eqref{eqn:NLH} on $(0, T) \times \R^{D}$. Let $T_+=T_+(u_0)>0$ denote the maximal time of existence.

The energy $E(u(t))$ is absolutely continuous and non-increasing as a function of $t \in[0, T_+)$ and for any $t_1, t_2 \in\left[0, T_{+}\right)$, $t_1\leq t_2$ there holds,
\begin{align}\label{eq:energy-identity}
 E(u(t_2))+ \int_{t_1}^{t_2} \int_0^{\infty}\left(\partial_t u(t, r)\right)^2 r^{D-1} \dr \mathrm{d} t=E(u(t_1)).
\end{align}
In particular, assuming boundedness of $\sup_{t\in [0,T_+)}\|u(t)\|_{\E}$ we have
\begin{align}\label{eq:tension-L2}
  \int_0^{T_{+}} \int_0^{\infty}\left(\partial_t u(t, r)\right)^2 r^{D-1} \mathrm{d} r \mathrm{d} t \lesssim \sup_{t\in [0,T_+)}(\|u(t)\|_{\E}^2+ \|u(t)\|_{\calE}^{2^*}) < \infty.
\end{align}
Furthermore, there exists $\rho>0$ such that if $u_0\in \dot{H}^1(\R^D)$ and $\|u_0\|_{\dot{H}^1}\leq \rho$, then $T_+=\infty$ and $\lim_{t\to \infty} \|u(t)\|_{\dot{H}^1}=0.$
\end{lem}
\subsection{Multi-Bubble Configuration}\label{subsec:multi-bubble}
Next, we recall some facts about solutions near multi-bubble configurations. For proofs and further references, see \cite{lawrie-wave}. Denote $\LL_{\calW}$ the operator obtained by linearization of~\eqref{eqn:NLH} about an $M$-bubble configuration $ \calW(\vec \iota, \vec \lam)$
\EQ{  \label{eq:LW-def} 
\LL_{\calW} \, g := \uD^2 E(\calW( \vec\iota, \vec \lam)) g = - \De g - f'(\calW( \vec\iota, \vec \lam) )g, 
}
where $f(z) := \abs{z}^{\frac{4}{D-2}} z$ and $f'( z) = \frac{D+2}{D-2} \abs{z}^{\frac{4}{D-2}}$. Given $g \in \E$, 
\EQ{
\La \uD^2 E(\calW(\vec \iota, \vec \lam)) g \mid g \Ra =  \int_0^\infty \Big((\p_r g(r))^2 -  f'(\calW( \vec \iota, \vec \lam)) g(r)^2 \, \Big) r^{D-1} \ud r. 
}
Denote $\mathcal{L}_\lambda:=-\Delta-f^{\prime}\left(W_\lambda\right)$ the linearization around a single bubble, $W_{\lam}$ and set $\LL:=\LL_{1}.$ Regarding the spectrum of $\LL$, ~\cite[Proposition 5.5]{DM08} showed that there exists a unique negative simple eigenvalue that we denote by $-\kappa^2<0$ (where $\kappa>0).$ Denote $\calY$ as the normalized (in $L^2$) eigenfunction associated to this eigenvalue. Fix any smooth function $\calZ \in C^{\infty}_c((0, \infty))$ such that the following holds 
\EQ{\label{eq:ZQ} 
  \ang{ \calZ \mid \Lam W} >0   \mand \ang{\calZ \mid \calY} = 0. 
 }
We record here the following energy expansion around a multi-bubble configuration whose proof is similar to the proof of Lemma 2.15 in \cite{lawrie-wave}.
\begin{lem}\label{lem:M-bub-energy} Let $M \in \N$. Then for any $\theta>0$, there exists a constant $\eta>0$ such that the following holds. Let $ \calW(\vec \iota, \vec \lam)$ be an $M$-bubble configuration such that
\EQ{
\sum_{j =1}^{M-1} \Big( \frac{ \lam_{j}}{\lam_{j+1}} \Big)^{\frac{D-2}{2}} \le \eta.
}
Then, 
\EQ{
  \Big|  E( \calW( \vec \iota, \vec \lam))  - M E(  W) +  \frac{(D(D-2))^{\frac{D}{2}}}{D} \sum_{j =1}^{M-1} \iota_j \iota_{j+1}  \Big( \frac{ \lam_{j}}{\lam_{j+1}} \Big)^{\frac{D-2}{2}}  \Big| \le \te \sum_{j =1}^{M-1} \Big( \frac{ \lam_{j}}{\lam_{j+1}} \Big)^{\frac{D-2}{2}} .
}
Furthermore, there exists a uniform constant $C>0$ such that for any $g \in \E$, 
\EQ{
\abs{\ang{\mathrm{D}E( \calW(\vec \iota, \vec \lam)) \mid g} } \le C \| g \|_{\E} \sum_{j =1}^M \Big( \frac{\lam_{j}}{\lam_{j+1}} \Big)^{\frac{D-2}{2}} . 
}
\end{lem} 
To measure how much we deviate from a multi-bubble configuration, we define the following proximity function.
\begin{defn} \label{def-d} Fix $M$ as in Definition~\ref{defn:multi-bubble} and let $ v \in \E$.  Define, 
\EQ{ \label{eq:d-def} 
\operatorname{dist}_M(  v) := \inf_{\vec \iota, \vec \lam}  \Big( \|  v -  \calW( \vec \iota, \vec \lam) \|_{\E}^2 + \sum_{j =1}^{M-1} \Big( \frac{\lam_{j}}{\lam_{j+1}} \Big)^{\frac{D-2}{2}} \Big)^{\frac{1}{2}},
}
where the infimum is taken over all vectors $\vec \lam = (\lam_1, \dots, \lam_M) \in (0, \infty)^M$ and all $\vec \iota = \{ \iota_1, \dots, \iota_M\} \in \{-1, 1\}^M$. 
\end{defn} 
Using a similar argument as in Lemma 2.17 in \cite{lawrie-wave}, when the value of the proximity function is small and the energy is close to the sum of energies of each bubble, we can find signs $\vec{\io}$ and scales $\vec{\lam}$ that realize the infimum in the above definition.
\begin{lem}\label{lem:mod-static} Let $M \in \N$. There exist constants $\eta, C>0$ such that the following hold. For any $ v \in  \calE$ satisfying
\EQ{ \label{eq:v-M-bub} 
\operatorname{dist}_M(  v)  \le \eta,
}
there exists a unique choice of parameters, $\vec \lam = ( \lam_1, \dots, \lam_M) \in  (0, \infty)^M$, $\vec\iota \in \{-1, 1\}^M$, and $g \in  \calE$, such that for all $j\in \{1,\cdots,M\}$,
\begin{align}
 v &=   \calW(\vec \iota, \vec \lam) +  g, \quad 
   0  = \La \calZ_{\U{\lam_j}} \mid g\Ra  \label{eq:v-decomp} \\  
\operatorname{dist}_M(  v)^2 &\le \|   g \|_{\E}^2  + \sum_{j =1}^{M-1} \Big( \frac{\lam_{j}}{\lam_{j+1}} \Big)^{\frac{D-2}{2}}  \le C \operatorname{dist}_M(v)^2  \label{eq:g-bound-0}.
\end{align}
\end{lem}  
Furthermore, similar to Lemma 2.20 in \cite{lawrie-wave}, if a function $w$ is close to two different multi-bubble configurations, then the scales of those two configurations are also the same up to a small constant. 
\begin{lem}\label{lem:bub-config} There exists $\eta>0$ sufficiently small with the following property. Let  $M, L \in \N$,  $\vec\iota \in \{-1, 1\}^M, \vec \sigma \in \{-1, 1\}^L$, $\vec \lam \in (0, \infty)^M, \vec \mu \in (0, \infty)^L$, and $w\in \E$ satisfying, 
 \begin{align} 
 \| w  -  \calW(  \vec \iota, \vec \lam)\|_{\E}^2  + \sum_{j =1}^{M-1} \Big(\frac{\lam_j}{\lam_{j+1}} \Big)^{\frac{D-2}{2}} \le \eta,  \quad \| w  -  \calW( \vec \sigma , \vec \mu)\|_{\E}^2 +  \sum_{j =1}^{L-1} \Big(\frac{\mu_j}{\mu_{j+1}} \Big)^{\frac{D-2}{2}} \le \eta.  
 \end{align} 
 Then, $M = L$, $\vec \iota = \vec \sigma$. Furthermore, for all $\te>0$, the parameter $\eta>0$ above can be chosen small enough so that
\EQ{ \label{eq:lam-mu-close} 
\max_{j = 1, \dots M} \left| \frac{\lam_j}{\mu_j} - 1 \right| \le  \te.
 }
\end{lem} 
\subsection{Localized Energy Inequalities and Energy Trapping}
Since our argument will rely on energy estimates, we record here some localized energy identities. 
\begin{lem}\label{lem: localized energy density}
Let $I \subset[0, \infty)$ be a time interval and let $\phi:$ $I \times(0, \infty) \rightarrow[0, \infty)$ be a smooth function. Let $u(t) \in \mathcal{E}$ be a solution to \eqref{eqn:NLH} on $I$. Then, for any $t_1,t_2 \in I$ with $t_1<t_2$ we have
\begin{align}\label{eqn:modified loc energy equality}
&\int_0^{\infty} \Tilde{\mathbf{e}}(u(t_2)) \phi^2 r^{D-1}\dr - \int_0^{\infty} \Tilde{\mathbf{e}}(u(t_1))\phi^2 r^{D-1}\dr \\
&= -2 \int_{t_1}^{t_2} \int_0^{\infty}(\partial_t u)^2 \phi^2 r^{D-1}\dr \mathrm{d}t +2 \int_{t_1}^{t_2}\int_0^{\infty} |u|^{p-1}u(\partial_t u) \phi^2 r^{D-1}\dr \mathrm{d}t  \\
&\quad - 4\int_{t_1}^{t_2} \int_0^{\infty} (\partial_r u) (\partial_t u) \phi \partial_r \phi \ r^{D-1}\dr \mathrm{d}t +2\int_{t_1}^{t_2} \int_0^{\infty} \frac{u\partial_t u}{r^2} \phi^2 r^{D-1}\dr \mathrm{d}t\\
   &\quad +2\int_{t_1}^{t_2}\int_0^{\infty}\Tilde{\mathbf{e}}(u(t)) \phi \partial_t \phi \ r^{D-1} \dr \mathrm{d}t.
\end{align}
In particular, if $\partial_t \phi(t,r)\leq 0$ we have
\begin{align}
&\int_0^{\infty} \Tilde{\mathbf{e}}(u(t_2)) \phi^2 r^{D-1}\dr - \int_0^{\infty} \Tilde{\mathbf{e}}(u(t_1)) \phi^2 r^{D-1}\dr \label{eqn:modified energy ineq I}\\
&\leq -\int_{t_1}^{t_2}\int_0^{\infty} (\partial_t u)^2  \phi^2 r^{D-1}\dr \mathrm{d}t + 4 \int_{t_1}^{t_2}\int_0^{\infty} |\partial_r u|^2|\partial_r \phi|^2  r^{D-1}\dr \mathrm{d}t  \\
&\quad +2 \left(\int_{t_1}^{t_2}\int_0^{\infty} |u|^{2p}  \phi^2 r^{D-1}\dr \mathrm{d}t\right)^{1/2}  \left(\int_{t_1}^{t_2}\int_0^{\infty} (\partial_t u)^2  \phi^2 r^{D-1}\dr \mathrm{d}t \right)^{1/2} \\
   &\quad +2 \left(\int_{t_1}^{t_2} \int_0^{\infty} \frac{|u|^2}{r^4}  \phi^2 r^{D-1}\dr \mathrm{d}t\right)^{1/2}  \left(\int_{t_1}^{t_2} \int_0^{\infty}(\partial_t u)^2  \phi^2 r^{D-1}\dr \mathrm{d}t\right)^{1/2},\\
   &\int_0^{\infty} \Tilde{\mathbf{e}}(u(t_2)) \phi^2 r^{D-1}\dr - \int_0^{\infty} \Tilde{\mathbf{e}}(u(t_1))\phi^2 r^{D-1}\dr \leq 4\int_{t_1}^{t_2}\int_0^{\infty} |u|^{2p} \phi^2 r^{D-1}\dr \mathrm{d}t  \label{eqn:absorbing-ineq} \\
&\quad + 4\int_{t_1}^{t_2} \int_0^{\infty} |\partial_r u|^2 
|\partial_r \phi|^2 r^{D-1}\dr \mathrm{d}t +4\int_{t_1}^{t_2} \int_0^{\infty} \frac{|u|^2}{r^4} \phi^2 r^{D-1}\dr \mathrm{d}t,
\end{align}
and
\EQ{
&\int_0^{\infty} \Tilde{\mathbf{e}}(u(t_2))  \phi^2 r^{D-1}\dr - \int_0^{\infty} \Tilde{\mathbf{e}}(u(t_1)) \phi^2 r^{D-1}\dr \leq -2 \int_{t_1}^{t_2}\int_0^{\infty} (\partial_t u)^2  \phi^2 r^{D-1}\dr  \mathrm{d}t  \label{eqn:modified energy ineq II} \\
&\quad + 4 \left(\int_{t_1}^{t_2}\int_0^{\infty} (\partial_t u)^2 \phi^2 (\partial_r \phi)^2 r^{D-1}\dr \mathrm{d}t\right)^{1/2}   \left(\int_{t_1}^{t_2}\int_0^{\infty} (\partial_r u)^2 r^{D-1}\dr \mathrm{d}t \right)^{1/2} \\
&\quad +2 \left(\int_{t_1}^{t_2}\int_0^{\infty} |u|^{2p} \phi^2 r^{D-1}\dr \mathrm{d}t\right)^{1/2}  \left(\int_{t_1}^{t_2}\int_0^{\infty} (\partial_t u)^2 \phi^2 r^{D-1} \dr \mathrm{d}t\right)^{1/2}  \\
&\quad +2 \left(\int_{t_1}^{t_2} \int_0^{\infty} \frac{|u|^2}{r^4} \phi^2 r^{D-1}\dr \mathrm{d}t\right)^{1/2}  \left(\int_{t_1}^{t_2} \int_0^{\infty}(\partial_t u)^2 \phi^2 r^{D-1}\dr \mathrm{d}t\right)^{1/2}.
}
\end{lem}
\begin{proof}
Observe that
\EQ{
&\int_0^{\infty} \Tilde{\mathbf{e}}(u(t_2)) \phi^2 r^{D-1} \dr  - \int_0^{\infty} \Tilde{\mathbf{e}}(u(t_1))\phi^2 r^{D-1} \dr = \int_{t_1}^{t_2}\int_0^{\infty} \partial_t(\Tilde{\mathbf{e}}(u(t)) \phi^2) \ r^{D-1} \dr \mathrm{d}t\\
   &= 2\int_{t_1}^{t_2}\int_0^{\infty} \left(\partial_r u \partial_t \partial_r u+ \frac{u\partial_t u}{r^2} \right)\phi^2 r^{D-1} \dr \mathrm{d}t  + 2\int_{t_1}^{t_2}\int_0^{\infty}\Tilde{\mathbf{e}}(u(t)) \phi \partial_t \phi \ r^{D-1} \dr \mathrm{d}t    \\
   &= -2 \int_{t_1}^{t_2}\int_0^{\infty} (\partial_t u)^2 \phi^2 r^{D-1}\dr \mathrm{d}t  +2 \int_{t_1}^{t_2}\int_0^{\infty} |u|^{p-1}u(\partial_t u) \phi^2 r^{D-1}\dr \mathrm{d}t  \\
   &\quad - 4\int_{t_1}^{t_2}\int_0^{\infty} (\partial_r u) (\partial_t u) \phi \partial_r \phi \ r^{D-1}\dr \mathrm{d}t +2\int_{t_1}^{t_2}\int_0^{\infty} \frac{u\partial_t u}{r^2} \phi^2 r^{D-1}\dr \mathrm{d}t \\
   &\quad +2\int_{t_1}^{t_2}\int_0^{\infty}\Tilde{\mathbf{e}}(u(t)) \phi \partial_t \phi \ r^{D-1} \dr \mathrm{d}t. }
Thus, we have proved \eqref{eqn:modified loc energy equality}. The inequalities \eqref{eqn:modified energy ineq I}, \eqref{eqn:absorbing-ineq}, and \eqref{eqn:modified energy ineq II} now follow by first dropping the last term in \eqref{eqn:modified loc energy equality} since $\partial_t \phi \leq 0$ and then applying Cauchy-Schwarz and Young's inequality in different ways.

\end{proof}
Next, we recall the well known radial Sobolev embedding that gives pointwise control of a radial function in the energy space $\cE.$
\begin{lem}[Radial Sobolev Embedding]
Let $v\in \E.$ Then for $R>0$ we have
\begin{align}\label{eqn:radial-sobolev}
|v(R)| \leq \frac{\|v\|_{\E(R)}}{R^{(D-2)/2}}.
\end{align}
\end{lem}
\begin{proof}
We recall the proof here for the reader's convenience. By density, we may assume $v$ has compact support. Thus, we get
\begin{align*}
R^{D-2} v^2(R) & =\int_R^{\infty}\left(-\frac{D-2}{r} v^2(r)-2 v \partial_r v\right) r^{D-2} \mathrm{~d} r \\
& \leq \int_R^{\infty} 2|v|\left|\partial_r v\right| r^{D-2} \mathrm{~d} r \leq \int_R^{\infty}\left(\frac{|v|^2}{r^2}+\left|\partial_r v\right|^2\right) r^{D-1} \mathrm{~d} r=\|v\|_{\mathcal{E}(R)}^2,
\end{align*}
which on re-arranging implies \eqref{eqn:radial-sobolev}. 
\end{proof}
\begin{rem}\label{rem:radial-sobolev}
Note that from the above argument if $v\in C^\infty_c((r_1,r_2))$ where $0<r_1<r_2\leq \infty$ then we also have
\EQ{
|v(r)|\leq \frac{1}{r^{(D-2)/2}}\|v\|_{\cE(r_1,r_2)}
}
for all $r\in (r_1,r_2).$
\end{rem}
The coercivity of the nonlinear energy $E$ plays an important role in the proof of Theorem~\ref{thm:main} since localized energy inequalities will imply smallness of the energy, which we would like to transfer to the $\calE$-norm to deduce smallness of $\calE$-norm. This would, in turn, help us prove that the distance function $\bfd(t)$ is small, which is ultimately what we are after. Thus we would like to compare $E(u)$ with $\|u\|_{\calE}.$ This is not always possible, however when $\|u\|_{\calE}$ is small then we can prove such an estimate. 
\begin{lem}[Coercivity of Nonlinear Energy for small $\cE$--norm]\label{lem:basic-trapping}
There exist constants $\delta, C>0$ with the following properties. Let $v\in \E$ be such that $\|v\|_{\cE}\leq \delta$. Then
\begin{align*}
E(v) \simeq \|v\|^2_{\mathcal{E}}.
\end{align*}
\end{lem}
\begin{proof}
The inequality $E(v)\leq \frac{1}{2}\|v\|^2_{\calE}$ follows from the definitions of $E$ and $\calE$ respectively. Thus, we will establish $C\|v\|^2_{\E} \leq E(v)$.

By Sobolev and Hardy's inequality, there exist constants $C_1, C_2>0$
\begin{align*}
\|\partial_r v\|_{L^2} &\leq \|v\|_{\E} \leq (1+C_1)  \|\partial_r v\|_{L^2},\\
\|  v \|^{2^*}_{L^{2^*}} &\leq C_2 \|\partial_r v\|_{L^2}^{2^*} \leq C_2  \|v\|_{\E}^{2^*}.
\end{align*}
Thus we get
\begin{align*}
E(v) &= \frac{1}{2}\int_0^{\infty} |\partial_r v|^2 r^{D-1}\dr - \frac{1}{2^{*}}\int_0^{\infty} |v|^{2^*} r^{D-1} \dr\\
&\geq \frac{1}{2(1+C_1)}\| v\|^2_{\E} - \frac{C_2}{2^{*}}\|v \|^{2^{*}-2}_{\E}\|v \|^{2}_{\E}\\
&\geq \left(\frac{1}{2(1+C_1)} - \frac{C_2}{2^{*}}\delta^{2^{*}-2}\right)\|v \|^{2}_{\E}.
\end{align*}
Therefore, choosing $\delta>0$ small enough, there exists a constant $C>0$ such that
\EQ{
E(v)\geq C \|v\|_{\cE}^2.
}
\end{proof}
From the above Lemma, we see that a small $\calE$-norm yields coercivity for the nonlinear energy. The following lemma shows that the previous lemma can be upgraded to deduce coercivity of the nonlinear energy when we have smallness of the $\E$-norm only on the tail region, i.e., when $r\in (R,\infty)$ for any $R>0.$
\begin{lem}[Trapping on Tails]\label{lem:localized-coercivity}
There exist constants $\delta, C>0$ with the following property. Let $v\in \E$ and $R>0$ be such that $ \|v\|_{\E(R)} \leq \delta$. Then
\begin{align*}
E(v;R,\infty)\geq C \|v\|_{\E(R)}^2.
\end{align*}    
\end{lem}
\begin{proof}
We want to show that there exists a constant $C>0$ such that
\begin{align*}
E(v;R,\infty) \geq C \|v\|_{\E(R)}^2
\end{align*}
which simplifies to 
\begin{align}
    \frac{1}{2}\int_R^\infty |\partial_r v|^2 r^{D-1}\dr  -\frac{1}{2^*}\int_R^\infty |v|^{2^*} r^{D-1}\dr  \geq C \int_R^\infty |\partial_r v|^2r^{D-1}\dr  + C \int_R^\infty \frac{v^2}{r^2}r^{D-1}\dr. \quad  \label{eqn:target}.
\end{align}
Using \eqref{eqn:radial-sobolev} and the smallness of $\cE$--norm, $\|v\|_{\E(R)}\leq \delta$, we can control the nonlinear term $|v|^{2^*}$,
\begin{align}
\frac{1}{2}\int_R^\infty |\partial_r v|^2 r^{D-1}\dr  -\frac{1}{2^*}\int_R^\infty |v|^{2^*} r^{D-1}\dr \geq \frac{1}{2}\int_R^\infty |\partial_r v|^2 r^{D-1}\dr  -\frac{C_1}{2^*}\delta^{2^*-2} \int_R^\infty \frac{v^2}{r^2} r^{D-1}\dr,
\end{align}
where $C_1=2^{2/(D-2)}.$ Therefore on re-arranging \eqref{eqn:target} we need to prove that
\begin{align*}
  \left(\frac{1}{2}-C\right)\int_R^\infty |\partial_r v|^2 r^{D-1}\dr  -\frac{C_1}{2^*}\delta^{2^*-2} \int_R^\infty \frac{v^2}{r^2} r^{D-1}\dr  \geq C \int_R^\infty \frac{v^2}{r^2} r^{D-1}\dr.
\end{align*}
Now we claim that
\begin{align*}
    \int_R^\infty \frac{v^2}{r^2}r^{D-1}\dr \leq C_3 \int_R^\infty |\partial_r v|^2 r^{D-1}\dr,
\end{align*}
where $C_3$ is the same constant as in the original Hardy's inequality and therefore only depends on the dimension $D.$ We prove this as follows. Let $k\geq 0$ then
\begin{align*}
   0 &\leq  \int_R^\infty \left(\frac{v}{r} + k\partial_r v\right)^2 r^{D-1}\dr\\
     &= \int_R^\infty \frac{v^2}{r^2} + k^2|\partial_r v|^2 + 2 k\frac{v}{r} \partial_r v r^{D-1}\dr \\
     &= \int_R^\infty \left(\frac{v^2}{r^2} + k^2|\partial_r v|^2 \right) r^{D-1}\dr + k\int_R^\infty \frac{\partial_r v^2}{r} r^{D-1}\dr\\
     &= \int_R^\infty \left(\frac{v^2}{r^2} + k^2 |\partial_r v|^2 \right) r^{D-1}\dr - (D-2)k \int_R^\infty \frac{v^2}{r^2} r^{D-1}\dr - v^2(R)kR^{D-2}\\
    &= (1-(D-2)k) \int_R^\infty \frac{v^2}{r^2} r^{D-1}\dr + k^2 \int_R^\infty |\partial_r v|^2 r^{D-1}\dr - kv^2(R)R^{D-2}.
\end{align*}
Set $k=\frac{2}{D-2}\geq 0$ and re-arrange the above inequality to get
\begin{align}
 \int_R^\infty \frac{v^2}{r^2} r^{D-1}\dr \leq \frac{4}{(D-2)^2} \int_R^\infty |\partial_r v|^2 r^{D-1}\dr.
\end{align}
Therefore,
\begin{align*}
   &\left(\frac{1}{2}-C\right)\int_R^\infty |\partial_r v|^2 r^{D-1}\dr  -\frac{1}{2^*}\int_R^\infty |v|^{2^*} r^{D-1}\dr\\
   &\geq\left(\frac{1}{2}-C\right)\int_R^\infty |\partial_r v|^2 r^{D-1}\dr  -\frac{C_1}{2^*}\delta^{2^*-2} \int_R^\infty \frac{v^2}{r^2} r^{D-1}\dr\\
   &\geq \left(\frac{C_3^{-1}}{2}-CC_3^{-1}\right)\int_R^\infty \frac{v^2}{r^2} r^{D-1}\dr  -\frac{C_1}{2^*}\delta^{2^*-2} \int_R^\infty \frac{v^2}{r^2} r^{D-1}\dr\\
   &= \left(\frac{C_3^{-1}}{2}-CC_3^{-1} -\frac{C_1 \delta^{2^*-2}}{2}\right) \int_R^\infty \frac{v^2}{r^2} r^{D-1}\dr \\
   &\geq C \int_R^\infty \frac{v^2}{r^2} r^{D-1}\dr
\end{align*}
provided 
\begin{align*}
    \delta^{2^*-2} < \frac{C_3^{-1}(1-2C)-2C}{C_1}
\end{align*}
where $C_3^{-1} = \frac{(D-2)^2}{4}$ and $C_1 = 2^{2/(D-2)}.$ Thus for instance if we choose $C =\frac{C_3^{-1}}{4(1+C_3^{-1})}$ then 
\begin{align}
    \frac{C_3^{-1}(1-2C)}{C_1} - \frac{2C}{C_1} = \frac{C_3^{-1}}{C_1}\left(1-2C\left(\frac{1+C_3^{-1}}{C_3^{-1}}\right)\right) = \frac{C_3^{-1}}{2C_1} 
\end{align}
and so we need $\delta^{2^*-2} < \frac{C_3^{-1}}{2C_1}$ to get a constant $C = \frac{C_3^{-1}}{4(1+C_3^{-1})}$ in the desired inequality \eqref{eqn:target}.
\end{proof}
\begin{lem}[Propagation of small localized $\mathcal{E}$ norm]\label{lem:prop-small-E}
There exist constants $\delta, C>0$ with the following property. Let ${u}(t) \in \mathcal{E}$ be a solution to \eqref{eqn:NLH} with initial data ${u}(0)={u}_0 \in \cE$ on the time interval $I=[0,T_+(u_0))$. Let $0<r_1<r_2<\infty$. Suppose that
\EQ{\label{eqn:small-init-energy}
\left\|{u}_0\right\|_{\mathcal{E}\left(r_1/2, 2r_2\right)} \leq \delta,\quad \sup_{t\in J}\|u(t)\|_{\cE} \lesssim 1,
}
where $J=I \cap [0,\delta r_1^2]$. Then,
\EQ{\label{eqn:small-later-energy}
\sup_{t\in J}\|{u}(t)\|_{\mathcal{E}\left(r_1, r_2\right)} \leq C \delta,
}
\end{lem}

\begin{proof}
Let $\phi(r)$ be a smooth cut-off function such that $\phi\equiv 1$ on $[r_1,r_2]$ and $\phi\equiv 0$ on $(0,r_1/2]\cup [2r_2,\infty).$ Then $\partial_t \phi = 0$ and thus by \eqref{eqn:absorbing-ineq} we have
\EQ{
&\int_0^{\infty} \Tilde{\mathbf{e}}(u(t)) \phi^2 r^{D-1}\dr \lesssim \int_0^{\infty} \Tilde{\mathbf{e}}(u(0))\phi^2 r^{D-1}\dr + \frac{t}{r_1^2} \sup_{s \in J}\int_0^{\infty} \Tilde{\mathbf{e}}(u(s)) \phi^2 r^{D-1}\dr 
}
where we estimate the nonlinear term by \eqref{eqn:radial-sobolev} and the second inequality in \eqref{eqn:small-init-energy} and we also use $|\partial_r \phi|^2\lesssim 1/r_1^2.$ Therefore,
\EQ{
\sup_{t\in J}\int_0^{\infty} \Tilde{\bfe}(u(t)) \phi^2 r^{D-1}\dr \leq C_1 \delta + C_2 \delta \sup_{t\in J}\int_0^{\infty} \Tilde{\mathbf{e}}(u(t)) \phi^2 r^{D-1}\dr.
}
Choosing $\delta \in (0,1)$ small enough allows us to absorb the second term on the right-hand side of the above inequality, yielding the desired estimate.
\end{proof}
\begin{lem}[Short time evolution close to $W$]\label{lem:short-time}
Let $\iota \in\{-1,1\}$. There exist $\delta_0>0$ and a function $\varepsilon_0:\left[0, \delta_0\right] \rightarrow[0, \infty)$ with $\varepsilon_0(\delta) \rightarrow 0$ as $\delta \rightarrow 0$ with the following properties. Let ${v}_0 \in \mathcal{E}$ and let ${v}(t)$ denote the unique solution to \eqref{eqn:NLH} with ${v}(0)={v}_0$. Let $\mu_0, T_0>0$ and suppose that
\EQ{
\left\|{v}_0-\iota {W}_{\mu_0}\right\|_{\mathcal{E}}+\frac{T_0^2}{\mu_0}=\delta \leq \delta_0.
}
Then, $T_0<T_{+}\left({v}_0\right)$ and
\EQ{
\sup _{t \in\left[0, T_0\right]}\left\|{v}(t)-\iota {W}_{\mu_0}\right\|_{\mathcal{E}}<\varepsilon_0(\delta).
}
\begin{proof}
By rescaling, we may assume $\mu_0=1$. The result is then a particular case of the local Cauchy theory, in particular, the continuity of the data to the solution map at ${W}$.    
\end{proof}
\end{lem}
\begin{lem}[Localized short time evolution close to $W$]\label{lem:loc-short-time}
Let $\iota \in\{-1,1\}$. There exist $\delta_0>0$ and a function $\varepsilon_0:\left[0, \delta_0\right] \rightarrow[0, \infty)$ with $\varepsilon_0(\delta) \rightarrow 0$ as $\delta \rightarrow 0$ with the following properties. 

Let ${u}_0 \in \mathcal{E}, T_0<T_{+}\left({u}_0\right)$, $T_0\leq \delta_0 r_1^2$. Let ${u}(t)$ denote the unique solution to \eqref{eqn:NLH} with ${u}(0)={u}_0$ and assume that $\sup_{t\in [0,T_0]}\|u(t)\|_{\cE}\lesssim 1.$ Consider, $\mu_0>0,0<r_1<r_2<\infty$ and suppose that
\EQ{\label{eqn:scale-assumption}
\left\|{u}_0-\iota {W}_{\mu_0}\right\|_{\mathcal{E}\left(r_1/2,2r_2 \right)}+\frac{T_0^2}{\mu_0}=\delta \leq \delta_0.
}
Then,
\EQ{
\sup_{t\in [0,T_0]}\left\|{u}(t)-\iota {W}_{\mu_0}\right\|_{\mathcal{E}\left(r_1,r_2\right)}<\varepsilon_0(\delta).
}
\end{lem}
\begin{proof}
Consider ${v}_0:=\phi {u}_0+(1-\phi) \iota {W}_{\mu_0}$, where $\phi$ is the same function cut-off function as in the proof of Lemma \ref{lem:prop-small-E}. Then 
\EQ{
\|v_0-\iota W_{\mu_0}\|_\E = \|\phi\left(u_0-\iota W_{\mu_0}\right)\|_{\E}\leq \|u_0-\iota W_{\mu_0}\|_{\E(r_1/2,2r_2)}\leq \delta.
}
For $\delta\leq \delta_0$, applying Lemma \ref{lem:short-time} we get
\EQ{\label{eqn:v is close to W}
\sup_{t\in [0,T_0]}\|v(t)-\io W_{\mu_0}\|_{\cE} \leq \veps_0(\delta).
}
 Since $u(t)-\io W_{\mu_0}=u(t)-v(t) + v(t)-\io W_{\mu_0},$ it suffices to estimate the difference $w(t)=u(t)-v(t).$ To this end  observe that $w$ solves
 \EQ{
 \partial_t w &= \Delta w + f\\
 w(0) &= (1-\phi)(u_0-\io W_{\mu_0}),
 }
where $f:=|u|^{p-1}u-|v|^{p-1}v$ and $p=2^*-1$. Then since $\partial_t \phi \equiv 0$ we have
\EQ{
&\int_0^{\infty} \Tilde{\mathbf{e}}(w(t)) \phi^2 r^{D-1} \dr  - \int_0^{\infty} \Tilde{\mathbf{e}}(w(0))\phi^2 r^{D-1} \dr = \int_{0}^{t}\int_0^{\infty} \partial_t(\Tilde{\mathbf{e}}(w(t)) \phi^2) \ r^{D-1} \dr \mathrm{d}t\\
   &= 2\int_{0}^{t}\int_0^{\infty} \left(\partial_r w \partial_t \partial_r w+ \frac{w\partial_t w}{r^2} \right)\phi^2 r^{D-1} \dr \mathrm{d}t  + 2\int_{0}^{t}\int_0^{\infty}\Tilde{\mathbf{e}}(w(t)) \phi \partial_t \phi \ r^{D-1} \dr \mathrm{d}t    \\
   &= -2 \int_{0}^{t}\int_0^{\infty} (\partial_t w)^2 \phi^2 r^{D-1}\dr \mathrm{d}t  +2 \int_{0}^{t}\int_0^{\infty} f(\partial_t w) \phi^2 r^{D-1}\dr \mathrm{d}t  \\
   &\quad - 4\int_{0}^{t}\int_0^{\infty} (\partial_r w) (\partial_t w) \phi \partial_r \phi \ r^{D-1}\dr \mathrm{d}t +2\int_{0}^{t}\int_0^{\infty} \frac{w\partial_t w}{r^2} \phi^2 r^{D-1}\dr \mathrm{d}t.
}
Therefore using Young's inequality and $|\partial_r \phi|\lesssim 1/r_1$ we get
\EQ{\label{eqn:integral-inequality}
 \int_0^{\infty} \Tilde{\mathbf{e}}(w(t)) \phi^2 r^{D-1} \dr 
 &\lesssim \int_0^{\infty} \Tilde{\mathbf{e}}(w(0))\phi^2 r^{D-1} \dr +\int_0^t \int_0^\infty |f|^{2} \phi^2 r^{D-1}\dr \ud s\\
 &\quad + \frac{t}{r_1^2} \sup_{0\leq s\leq t} \int_0^\infty \Tilde{\mathbf{e}}(w(s)) \phi^2 r^{D-1}\dr\\
 &\leq C_1\delta  + C_2 \delta \sup_{0\leq s\leq t} \int_0^\infty \Tilde{\mathbf{e}}(w(s)) \phi^2 r^{D-1}\dr,
}
where in the last inequality, we estimate $f$ as follows
\EQ{
r^4 |f|^2 &= r^4 ||u|^{p-1}u-|v|^{p-1}v|^2 \\
&\leq  r^4 |u-v|^2 ||u|^{p-1} +|v|^{p-1}|^2 \\
&=  |w|^2 | r^2 |u|^{p-1} +r^2 |v|^{p-1}|^2 \\
&\lesssim  |w|^2. 
}
To obtain the last inequality in the above display we use \eqref{eqn:radial-sobolev} to estimate $r^2 |u|^{p-1}$ and $r^2 |v|^{p-1}$ by $\|u(t)\|_{\cE}$ and $\|v(t)\|_{\cE}$ respectively for all $t\in (0, T_0]$. The first quantity is bounded by assumption, while the second quantity can be estimated by \eqref{eqn:v is close to W} and $T_0^2/\mu_0 \leq \delta.$ Thus, for $\delta\leq \delta_0$ small enough, we can absorb the second term in the RHS of \eqref{eqn:integral-inequality} to deduce that $\|w(t)\|_{\cE(r_1,r_2)}\lesssim \delta$ for all $t\in (0, T_0]$. Therefore,  
\EQ{
\|u(t)-\io W_{\mu_0}\|_{\cE(r_1,r_2)} \lesssim \|w(t)\|_{\cE(r_1,r_2)}+\|v(t)-\io W_{\mu_0}\|_{\cE(r_1,r_2)} \leq \veps_0(\delta),
}
which gives us the desired result.
\end{proof}
\begin{lem}\label{lem:un-seq}
If $\iota_n \in\{-1,0,1\}, 0<t_n^2 < r_n \ll \mu_n\ll R_n$ and $\{u_n\}_{n\in \N}$ is a sequence of solutions of \eqref{eqn:NLH} such that ${u}_n(t)$ is defined for $t \in\left[0, t_n\right]$, $\sup_{t\in [0,t_n]}\|u_n(t)\|_{\calE}\lesssim 1$ and
\EQ{
\lim _{n \rightarrow \infty}\left\|{u}_n(0)-\iota_n {W}_{\mu_n}\right\|_{\mathcal{E}\left(r_n/2, 2R_n\right)}=0,
}
then
\EQ{
\lim _{n \rightarrow \infty} \sup _{t \in\left[0, t_n\right]}\left\|{u}_n(t)-\iota_n {W}_{\mu_n}\right\|_{\mathcal{E}\left(r_n, R_n\right)}=0 .
}
\end{lem}
\begin{proof}
This is a direct consequence of Lemma \ref{lem:prop-small-E} when $\iota_n=0$ and Lemma \ref{lem:loc-short-time} when $\iota_n \in$ $\{-1,1\}$.
\end{proof}

\section{Sequential bubbling for finite time blow-up solutions}\label{sec:finite-time}
The goal of this section is to establish sequential soliton resolution when $T_+<\infty.$ To this end, we first establish a lemma that shows strong convergence to the weak limit outside the origin and convergence of the nonlinear energy on compact sets near the origin.
\begin{lem}\label{lem:body-map}
Let $u_0 \in \mathcal{E}$ and let $u(t)$ be a solution to \eqref{eqn:NLH} with $u(0)=u_0$. Suppose that $T_{+}=T_+\left(u_0\right)<\infty$, and $\sup_{t\in[0,T_+)}\|u(t)\|_{\cE}<\infty$. Then, there exists a unique function $u^* \in \mathcal{E}$ such that for any $r_0>0$,
\begin{align}\label{eqn:smooth-conv-out-origin}
\lim _{t \rightarrow T_+}\left\|u(t)-u^*\right\|_{\mathcal{E}\left( r_0\right)}=0 .  
\end{align}
Moreover, there exists a finite constant $L\in \R$ such that for each $r_0 \in(0, \infty]$,
\begin{align}\label{eqn:finite-num-bubble}
\lim _{t \rightarrow T_{+}} \int_0^{r_0} \bfe(u(t)) r^{D-1}\ud r &=L+\int_0^{r_0} \bfe(u^*) r^{D-1}\ud r 
\end{align}
In particular,
\EQ{
\lim_{r_0\to 0}\lim _{t \rightarrow T_{+}} \int_0^{r_0} \bfe(u(t)) r^{D-1}\ud r &=L.
}
\end{lem}
\begin{rem}
See the works of Struwe \cite{struwe} and Qing \cite{qing} for analogous results for the harmonic map heat flow.
\end{rem}

\begin{proof}[Proof of Lemma \ref{lem:body-map}]
We first show \eqref{eqn:smooth-conv-out-origin}. Define the blow-up or the singular set
\begin{align}\label{eqn:sing-set}
\mathcal{S}=\left\{x \in \R^D: \lim _{R \rightarrow 0} \liminf _{t \rightarrow T_{+}} \int_{B_R(x) }\left|u(t)\right|^{p+1} \ud y \geq \varepsilon_0\right\},
\end{align}
where $\varepsilon_0>0$ is a constant coming from the $\varepsilon$--regularity established in Theorem 2.1 in \cite{duShi}. Observe that when $u$ is radially symmetric then $\mathcal{S}=\{0\}$ since by radial Sobolev embedding \eqref{eqn:radial-sobolev} and $\sup_{t\in [0,T_+)}\|u(t)\|_{\cE}<\infty$ we have that $u\in L^\infty$ on the region $(r_0,\infty)\times [0,T_+)$ for any $r_0>0$. By the standard parabolic regularity, we have convergence in $\dot{H}^1 \cap C^0$ of $u(t)$ away from the origin as $t \rightarrow T_{+}$. Let $u^*$ be this limit. Then $u^*$ is a continuous function on $\mathbb{R}^D \backslash\{0\}$. Since $u(t)$ is $\dot{H}^1$-bounded, the $\dot{H}^1 \cap C^0$-convergence away from the origin also implies $u^* \in \dot{H}^1$ and $u(t) \rightharpoonup u^*$ in $\dot{H}^1$. By the uniqueness of the weak limit, $u^* \in \dot{H}^1$ is unique.

Furthermore, since $u\in L^\infty$ on the region $[0,T_+)\times (r_0,\infty)$ for any $r_0>0$, standard parabolic regularity theory implies $\dot{H}^1\cap C^0$ convergence to a weak limit $u^*$ away from the origin, implying \eqref{eqn:smooth-conv-out-origin}. Next to see \eqref{eqn:finite-num-bubble}, we argue as in the proof of Proposition 2.1 in \cite{qing}. Denote,
\begin{align}
\limsup _{t \rightarrow T_+} \int_{0}^{r_0} \bfe(u(t)) r^{D-1}\mathrm{d}r &= \lim _{t_i \rightarrow T_+} \int_0^{r_0} \bfe(u(t_i)) r^{D-1}\mathrm{d}r \\
&= M+\int_0^{r_0} \bfe(u^*) r^{D-1}\ud r  \label{eqn:limsup}\\
\liminf_{t \rightarrow T_+} \int_{0}^{r_0} \bfe(u(t)) r^{D-1}\mathrm{d}r &= \lim _{s_i \rightarrow T_+} \int_0^{r_0}\bfe(u(s_i)) r^{D-1}\mathrm{d}r \\
&= m+\int_0^{r_0} \bfe(u^*) r^{D-1}\ud r  \label{eqn:liminf}
\end{align}
It suffices to show $M=m.$ Note that $M$ and $m$ are independent of $r_0$ by \eqref{eqn:smooth-conv-out-origin}. Thanks to \eqref{eqn:smooth-conv-out-origin} we can also assume that $r_0<\infty$. Take $\lambda_j=\frac{r_0}{j} $. Then observe that
\EQ{\label{eqn:lambda_j-energy-expansion}
&\int_0^{\lam_j} \bfe(u(t)) r^{D-1}\ud r = \int_0^{r_0} \bfe(u(t)) r^{D-1}\ud r -\int_{\lam_j}^{r_0} \bfe(u(t)) r^{D-1}\ud r\\
&\quad \geq \int_0^{r_0} \bfe(u(t)) r^{D-1}\ud r -\int_{\lam_j}^{r_0} |\bfe(u(t)) -\bfe(u^*)| r^{D-1}\ud r - \int_{\lam_j}^{r_0} \bfe(u^*) r^{D-1}\ud r.
}
Then since $u(t) \rightarrow u^*$ smoothly outside the origin as $t \rightarrow T_+$, there exist subsequences still denoted by $\{t_{i_j}\}$ and $\{s_{i'_j}\}$ such that
\EQ{\label{eqn:t_i-weak-lim}
\int_{\lam_j}^{r_0} |\bfe(u(t_{i_j})) -\bfe(u^*)| r^{D-1}\ud r \to 0,\quad \text{ as }j\to \infty
}
and 
\EQ{\label{eqn:s_j-weak-lim}
\int_{\lam_j}^{r_0} |\bfe(u(s_{i'_j})) -\bfe(u^*)| r^{D-1}\ud r \to 0,\quad \text{ as }j\to \infty.
}
Denote $t_j=t_{i_j}$ and $s_j=s_{i'_j}$. Up to subsequences we can arrange that $s_j \leq t_j \leq s_{j+1} \leq t_{j+1}$. Then \eqref{eqn:t_i-weak-lim}, \eqref{eqn:lambda_j-energy-expansion} and \eqref{eqn:limsup} implies
\EQ{
\int_0^{\lam_j} \bfe(u(t_{j})) r^{D-1}\ud r &\geq M+\int_0^{r_0} \bfe(u^*)r^{D-1}\ud r - \int_{\lam_j}^{r_0} \bfe(u^*) r^{D-1}\ud r-o_j(1) \\
&\geq M+\int_0^{\lam_j}\bfe(u^*) r^{D-1}\ud r -o_j(1),
}
where $o_j(1)\to 0$ as $j\to \infty.$
Thus,
\EQ{\label{eqn:limsup-seq-lb}
\int_0^{\lam_j}\bfe(u(t_j)) r^{D-1}\mathrm{d} r \geq M-o_j(1).
}
Next, we derive a localized energy inequality for the localized nonlinear energy density $\bfe(u).$ Let $\phi \in C^\infty_c([0,\infty)).$ Then observe that
\begin{align}\label{eqn:modified loc dirichlet equality}
&\int_0^{\infty} {\bfe}(u(t_2)) \phi^2 r^{D-1}\dr - \int_0^{\infty} {\bfe}(u(t_1))\phi^2 r^{D-1}\dr \\
&=  -\int_{t_1}^{t_2} \int_{0}^\infty (\partial_t u)^2\phi^2 r^{D-1}\dr \ud t-2\int_{t_1}^{t_2}\int_{0}^\infty (\partial_r u)(\partial_r \phi) \phi (\partial_t u)  r^{D-1}\dr \ud t.
\end{align}
Then using the above identity with $t_1=s_j, t_2=t_j$ and $\phi$ such that $\phi\equiv 1$ on $[0,\lam_j)$ and $\phi\equiv 0$ on $(\delta,\infty)$ for any $\delta \in (\lam_j,r_0)$ with $|\partial_r \phi|\lesssim (\delta-\lam_j)^{-1}$ we get
\EQ{
\int_0^{\delta} \bfe(u(s_j)) \phi^2 r^{D-1}\dr &\geq \int_0^{\delta} \bfe(u(t_j)) \phi^2 r^{D-1}\dr + \int_{s_j}^{t_j} \int_0^{\infty}(\partial_t u)^2 \phi^2 r^{D-1}\dr \mathrm{d}t \\
&\quad + 2\int_{s_j}^{t_j} \int_0^{\infty} (\partial_r u) (\partial_t u) \phi \partial_r \phi \ r^{D-1}\dr \mathrm{d}t\\
&\geq \int_0^{\lam_j} \bfe(u(t_j)) r^{D-1}\dr +  \int_{\lam_j}^{\delta} \bfe(u(t_j)) \phi^2 r^{D-1}\dr  - \int_{s_j}^{t_j} \int_{\lam_j}^{\delta} |\partial_r u|^2 |\partial_r \phi|^2 r^{D-1}\dr \mathrm{d}t\\
&\geq \int_0^{\lam_j} \bfe(u(t_j)) r^{D-1}\dr -  \int_{\lam_j}^{\delta} |\bfe(u(t_j))-\bfe(u^*)| \phi^2 r^{D-1}\dr \\
&\quad +\int_{\lam_j}^{\delta} \bfe(u^*) \phi^2 r^{D-1}\dr- C_1 \frac{(t_j-s_j)}{(\delta-\lam_j)^2}\\
&\geq \int_0^{\lam_j} \bfe(u(t_j)) r^{D-1}\dr-o_j(1) - o_\delta(1) -  C_1 \frac{(t_j-s_j)}{(\delta-\lam_j)^2},
}
where we used the fact that 
\begin{align}\label{eqn:(**)}
     \int_{\lam_j}^{\delta} |\bfe(u(t_j))-\bfe(u^*)| \phi^2 r^{D-1}\dr &\to 0,\text{ as }j\to \infty\\
     \left|\int_{\lam_j}^{\delta} \bfe(u^*)\right|\lesssim \int_0^\delta ((\partial_r u^*)^2 + |u^*|^{\frac{2D}{D-2}})  r^{D-1} \dr &\to 0,\text{ as } \delta\to 0,\text{ and }\\
     \int_{s_j}^{t_j} \int_{\lam_j}^{\delta} |\partial_r u|^2 |\partial_r \phi|^2 r^{D-1}\dr \mathrm{d}t &\leq C_1 \frac{(t_j-s_j)}{(\delta-\lam_j)^2}
\end{align}
for some constant $C_1>0$ depending on $\sup_{t\in [0,T_+)} \|u(t)\|_{\cE}.$ Therefore,
\EQ{\label{eqn:liminf-to-limsup-prop}
\int_0^{\delta} \bfe(u(s_j)) r^{D-1}\dr &= \int_0^\delta \bfe(u(s_j))\phi^2 r^{D-1}\dr + \int_{\lam_j}^\delta \bfe(u(s_j))(1-\phi^2)r^{D-1}\dr \\
&\geq \int_0^{\lam_j} \bfe(u(t_j)) r^{D-1}\dr-o_j(1) - o_\delta(1)-  C_1 \frac{(t_j-s_j)}{(\delta-\lam_j)^2},
}
where we use \eqref{eqn:s_j-weak-lim} and a similar argument as above to get that
\EQ{
\int_{\lam_j}^\delta \bfe(u(s_j))(1-\phi^2)r^{D-1}\dr \geq \int_{\lam_j}^\delta \bfe(u^*)(1-\phi^2)r^{D-1}\dr-o_j(1) \geq -o_{\delta}(1)-o_j(1).
}
Consequently, sending $s_j \to T_+$ and recalling  in \eqref{eqn:liminf-to-limsup-prop} and using \eqref{eqn:limsup-seq-lb} we get
\EQ{
\lim _{s_j \rightarrow T_+} \int_0^{r_0} \bfe(u(s_j)) r^{D-1}\mathrm{d}r &= \lim _{s_j \rightarrow T_+}\int_\delta^{r_0} \bfe(u(s_j)) r^{D-1}\mathrm{d}r+ \lim _{s_j \rightarrow T_+}\int_0^\delta  \bfe(u(s_j)) r^{D-1}\mathrm{d}r \\
&\geq M- o_{\delta}(1) + \int_\delta^{r_0} \bfe(u^*) r^{D-1}\mathrm{d} r.
}
Taking $\delta \rightarrow 0$, and recalling \eqref{eqn:liminf}, we can conclude that $m=M$. Thus, the limit $L=M$ exists in \eqref{eqn:finite-num-bubble}. Furthermore, $L$ is finite due to the fact that $\sup_{t\in [0,T_+)}\|u(t)\|_{\calE}<\infty.$

\end{proof}
\begin{prop}[Sequential bubbling for solutions that blow up in finite time]  \label{prop:seq-ftbu} 
Let $u_0 \in \E$, and let $u(t)$ denote the solution to~\eqref{eqn:NLH} with initial data $u_0$ with bounded $\calE$-norm. Suppose that $T_+(u_0) < \infty$. Then, there exist a unique function $u^* \in \E$,  a unique integer $N \ge 1$, a sequence of times $t_n \to T_+$, signs $\vec\iota \in \{-1, 1\}^N$, a sequence of scales $\vec \lam_n \in (0, \infty)^N$, and an error $g_n$ defined by 
\EQ{
u(t_n) = \sum_{j =1}^N \iota_j W_{\lam_{n,j}} + u^* + g_n,
}
with the following properties: 
\begin{itemize} 
\item[(i)] The integer $N \ge 1$ and a weak limit $u^*$ satisfy, 
\EQ{ \label{eq:energy-limit} 
\lim_{t \to T_+} E( u(t)) = N E(W)  + E( u^*); 
}
\item[(ii)] for any $0<\al<A$ 
\begin{align}
&\lim_{t \to T_+} E( u(t); 0, \al(T_+ - t)^{\frac{1}{2}}) = NE(W), \label{eq:N-bubbles-bu} \\
&\lim_{t\to T_+} E(u(t); \al (T_+ -t)^{\frac{1}{2}}, A(T_+ -t)^{\frac{1}{2}}) = 0,
\label{eq:en-annulus}\\
&\lim_{t \to T_+} E( u(t) - u^*; \al(T_+ - t)^{\frac{1}{2}}, \infty) = 0, \label{eq:en-ext-bu}     
\end{align}
and there exist $0< T_0 < T_+$ and function $\rho : [T_0, T_+) \to (0,\infty)$ satisfying, 
\begin{equation} \label{eq:radiation} 
\lim_{t \to T_+} \Big( \frac{\rho(t)}{\sqrt{T_+-t}} + \| u(t) -  u^*\|_{\cE(\rho(t))}\Big) = 0; 
\end{equation}
\item[(iii)] 
the error $g_n$ and the scales $\vec \lam_n$ satisfy, 
\EQ{ \label{eq:d(t_n)} 
\lim_{ n \to \infty} \Big( \| g_n \|_{\E}^2 + \sum_{j =1}^N \Big( \frac{\lam_{n, j}}{\lam_{n, j+1}} \Big)^{\frac{D-2}{2}}  \Big)^{\frac{1}{2}}  = 0 , 
}
where here we adopt the convention that $\lam_{n, N+1} := (T_+ - t_n)^{\frac{1}{2}}$. 
\end{itemize} 
\end{prop} 
\begin{proof}
Let $u(t) \in \E$ be a solution to \eqref{eqn:NLH} blowing up at time $T_+ \in (0,\infty)$. By ~\eqref{eq:tension-L2} we can find a sequence $t_n \to T_+$ so that, 
\EQ{
(T_+ - t_n)^{\frac{1}{2}} \| \calT( u(t_n)) \|_{L^2}  \to 0 \mas n \to \infty. 
}
Applying Lemma~\ref{lem:compactness} with $\rho_n:= (T_+ - t_n)^{\frac{1}{2}}$, yields $N \ge 0$, $\vec \iota \in \{-1, 1\}^N,$ and $\vec \lam_n \in (0, \infty)^N$ such that after passing to a subsequence, we have 
\EQ{ \label{eq:decomp-0} 
\lim_{n \to \infty} \Big( \| u(t_n) - \calW(\vec \iota, \vec \lam_n) \|_{\E( r \le A (T_+- t_n)^{\frac{1}{2}})}^2 + \sum_{j =1}^{N-1} \Big( \frac{ \lam_{n, j}}{\lam_{n, j+1}} \Big)^{\frac{D-2}{2}} \Big) = 0
}
for each $A>0$, and $\lam_{n, N} \lesssim (T_+- t_n)^{\frac{1}{2}}$ if $N\geq 1.$ Consider a smooth cut-off function
\begin{align}
\phi &\equiv 1\text{ on } [r_1,r_2],\quad \phi \equiv 0 \text{ on }(0,r_1/2]\cap[2r_2,\infty),
\end{align} 
where we will choose positive parameters $r_1$ and $r_2$ appropriately. Define the localized $\E$-norm on the annulus
\begin{align}
    \Tilde{\Theta}_{[r_1,r_2]}(t) = \int_0^\infty  \Tilde{\bfe}(u(t,r)) \phi(r)^2r^{D-1} \dr,\quad \Tilde{\Theta}^*_{[r_1,r_2]} = \int_0^\infty \Tilde{\bfe}(u^*(r)) \phi(r)^2 r^{D-1} \dr.
\end{align}
Using~\eqref{eqn:modified energy ineq II}, $|\partial_r \phi|\lesssim \frac{1}{r_1}$, and \eqref{eqn:radial-sobolev} for each $0<s<\tau< T_+$ we have
\EQ{ \label{eq:Theta-ineq} 
&| \Tilde{\Theta}_{[r_1,r_2]}(\tau) - \Tilde{\Theta}_{[r_1,r_2]}(s) | \\
&\lesssim  \int_{s}^{\tau} \|\partial_t u\|^2_{L^2}\mathrm{d}t+ \left(\int_{s}^{\tau} \int_0^{\infty}|u|^{2p} \phi^2\right)^{\frac{1}{2}}  \left(\int_{s}^{\tau} \int_0^{\infty} (\partial_t u)^2 \phi^2\right)^{\frac{1}{2}}  \\
&\quad + \left(\int_{s}^{\tau} \int_0^{\infty} (\partial_t u)^2 \phi^2 (\partial_r \phi)^2\right)^{\frac{1}{2}}   \left(\int_{s}^{\tau} \int_0^{\infty} (\partial_r u)^2 \right)^{\frac{1}{2}} + \left(\int_{s}^{\tau} \int_0^{\infty} \frac{|u|^2}{r^4} \phi^2\right)^{\frac{1}{2}}  \left(\int_{s}^{\tau} \int_0^{\infty}(\partial_t u)^2 \phi^2\right)^{\frac{1}{2}}\\
&\lesssim  \int_s^{T_+} \| \p_t u(t) \|_{L^2}^2 \, \ud t  + \left(\int_{s}^{\tau} \int_{r_1}^{r_2} \frac{1}{r^3}\dr \mathrm{d}t\right)^{\frac{1}{2}} \left(\int_{s}^{\tau} \|\partial_t u\|^2_{L^2} \mathrm{d}t\right)^{\frac{1}{2}} \\
&\quad + \frac{(\tau-s)^{\frac{1}{2}}}{r_1} \left(\int_{s}^{\tau} \|\partial_t u\|^2_{L^2} \mathrm{d}t\right)^{\frac{1}{2}} +\left(\int_{s}^{\tau} \int_{r_1}^{r_2} \frac{1}{r^{3}} \dr\ud t\right)^{\frac{1}{2}}  \left(\int_{s}^{\tau} \|\partial_t u\|^2_{L^2}\mathrm{d}t\right)^{\frac{1}{2}}\\
&\lesssim  \int_s^{T_+} \| \p_t u(t) \|_{L^2}^2 \, \ud t  + \frac{(T_+-s)^{\frac{1}{2}}}{r_1} \left(\int_{s}^{T_+} \|\partial_t u\|^2_{L^2} \mathrm{d}t\right)^{\frac{1}{2}}.
} 
Letting $s\to T_+$ we see that $\lim_{s\to T_+} \Tilde{\Theta}_{[r_1,r_2]}(s)$ exists. We first prove
\begin{align}\label{eqn:fixed-annulus-estimate}
\lim_{t\to T_+}\Tilde{E}(u(t);\alp(T_+-t)^{\frac{1}{2}},r_0) = \Tilde{E}(u^*;0,r_0)
\end{align}
for any $r_0 \in (0, \infty]$ and $\alp>0$. Let $r_2=r_0$, then 
\begin{align}\label{eqn:theta-diff}
\Tilde{\Theta}_{[r_1,r_0]}(\tau) - \Tilde{\Theta}^*_{[r_1,r_0]} = \int_{r_1/2}^{2r_0}  \left(\Tilde{\mathbf{e}}(u(\ta)) - \Tilde{\mathbf{e}}(u^*)\right)\phi(r)^2 r^{D-1}\ud r.
\end{align}
As $\tau\to T_+$, the expression on the RHS tends to zero by \eqref{eqn:smooth-conv-out-origin}. Thus choosing $r_1 = \alp(T_+-s)^{\frac{1}{2}}$ and $r_2=r_0$  in \eqref{eq:Theta-ineq} and using \eqref{eqn:theta-diff} we get
\begin{align}
    \left|\Tilde{\Theta}^*_{[\alp(T_+-s)^{\frac{1}{2}},r_0]}-\Tilde{\Theta}_{[\alp(T_+-s)^{\frac{1}{2}},r_0]}(s)\right| \lesssim \int_s^{T_+} \| \p_t u(t) \|_{L^2}^2 \, \ud t  + \frac{1}{\alp} \left(\int_{s}^{T_+} \|\partial_t u\|^2_{L^2} \mathrm{d}t\right)^{\frac{1}{2}},
\end{align}
which implies
\begin{align}\label{eqn:body-map-limit-r0}
   \lim_{s\to T_+} \Tilde{\Theta}_{[\alp(T_+-s)^{\frac{1}{2}},r_0]}(s) =  \Tilde{\Theta}^*_{[0,r_0]}.
\end{align}
By the same argument, if we take $r_1=\alp(T_+-s)^{\frac{1}{2}}/2$ and $r_2=\alp(T_+-s)^{\frac{1}{2}}$ we see that
\begin{align}
   \lim_{s\to T_+} \Tilde{\Theta}_{[\alp(T_+-s)^{\frac{1}{2}}/2,\alp(T_+-s)^{\frac{1}{2}}]}(s) = 0,
\end{align}
which implies 
\EQ{\label{eqn:ann}
\lim_{s\to T_+} \tilde{E}(u(s);\alp(T_+-s)^{\frac{1}{2}}/2,\alp(T_+-s)^{\frac{1}{2}})=0.
}
Since
\EQ{
&|\Tilde{E}(u(s); \alp(T_+-s)^{\frac{1}{2}},r_0) - \Tilde{E}(u^*; 0,r_0)| \\
&\lesssim |\Tilde{\Theta}_{[\alp (T_+-s)^{\frac{1}{2}},r_0]}(s) - \Tilde{\Theta}^*_{[\alp (T_+-s)^{\frac{1}{2}},r_0]}| + \Tilde{E}(u(s); \alp(T_+-s)^{\frac{1}{2}}/2,\alp(T_+-s)^{\frac{1}{2}}) \\
&\quad + \Tilde{E}(u^*; 0,\alp(T_+-s)^{\frac{1}{2}}) + |\int_{r_0}^{2r_0}  (\Tilde{\bfe}(u(s))-\Tilde{\bfe}(u^*))\phi^2r^{D-1}\mathrm{d}r|
}
and all four terms of the RHS of the above display go to zero as $s\to T_+$ by \eqref{eqn:body-map-limit-r0}, \eqref{eqn:ann}, $u^*\in \calE$, and \eqref{eqn:smooth-conv-out-origin} respectively, we get \eqref{eqn:fixed-annulus-estimate}.
\newline 
Proof of \eqref{eq:radiation}: It suffices to show
\begin{align}\label{eqn:tilde-e-ext}
   \lim_{\ta\to T_+} \Tilde{E}( u(\tau) - u^*;& \al (T_+- \tau)^{\frac{1}{2}}, \infty)=0.
\end{align}
for any $\alp>0$. Let $\veps >0$. Then note the following estimate,
\EQ{
\Tilde{E}( u(\tau) - u^*;& \al (T_+- \tau)^{\frac{1}{2}}, \infty) \le \Tilde{E}( u(\tau) - u^*; \al (T_+- \tau)^{\frac{1}{2}}, r_0) + \Tilde{E}( u(\tau) - u^*; r_0, \infty)\\
&\le 2\Tilde{E}( u(\tau); \al (T_+- \tau)^{\frac{1}{2}}, r_0) + 2 \Tilde{E}(u^*; \al (T_+- \tau)^{\frac{1}{2}}, r_0) + \Tilde{E}( u(\tau) - u^*; r_0, \infty).
}
We can first choose $r_0>0$ small enough such that for $\ta>0$ sufficiently close to $T_+$ we have
\begin{align}
    \Tilde{E}(u^*;\alp(T_+-\tau)^{\frac{1}{2}},r_0) \leq \Tilde{E}(u^*;0,r_0) \leq \veps.
\end{align}
Next using \eqref{eqn:smooth-conv-out-origin} we see that for $\tau$ sufficiently close to $T_+$
\begin{align}
\Tilde{E}(u(\tau) - u^*; r_0, \infty) \le \veps.
\end{align}
Finally from \eqref{eqn:fixed-annulus-estimate} we have that
\EQ{
 \Tilde{E}(u(\tau);\alp(T_+-\tau)^{\frac{1}{2}},r_0) \le 2\veps.
}
Thus, combining all the above estimates for $\tau$ sufficiently close to $T_+$, we have
\begin{align}
   \Tilde{E}( u(\tau) - u^*;& \al (T_+- \tau)^{\frac{1}{2}}, \infty) \le 4\veps
\end{align}
which establishes \eqref{eqn:tilde-e-ext} and thus yields \eqref{eq:radiation}. 
\newline 
Proof of \eqref{eq:en-ext-bu}: To show that the nonlinear energy vanishes on the region $(\alp(T_+-t)^{\frac{1}{2}},\infty)$ as $t\to T_+$, first recall that
\begin{align}
    E(u(t);\alp(T_+-t)^{\frac{1}{2}},\infty) =\frac{1}{2} \int_{\alp(T_+-t)^{\frac{1}{2}}}^\infty |\partial_r u(t)|^2 r^{D-1}\ud r - \frac{1}{2^*} \int_{\alp(T_+-t)^{\frac{1}{2}}}^{\infty} |u(t)|^{2^*}r^{D-1}\ud r.
\end{align}
The first term in the RHS of the above equality tends to zero due to \eqref{eqn:tilde-e-ext}. For the second term, using \eqref{eqn:radial-sobolev} we get
\begin{align}
  0\leq  \int_{\alp(T_+-t)^{\frac{1}{2}}}^{\infty} |u(t)|^{2^*} r^{D-1}\dr &= \int_{\alp(T_+-t)^{\frac{1}{2}}}^{\infty} r^2|u(t)|^{2^*-2} \frac{|u(t)|^2}{r^2} r^{D-1}\dr \\
  &\leq \|u(t)\|_{\cE(\alp(T_+-t)^{\frac{1}{2}})}^{2^*}\to 0,
\end{align}
as $t\to T_+$ due to \eqref{eqn:tilde-e-ext}. This proves \eqref{eq:en-ext-bu}. 
\newline
{Proof of \eqref{eq:en-annulus}}: Let $0< \al< A < \infty.$ We first show that 
\begin{align}\label{eqn:modified-energy-annulus-estimate}
\lim_{s \to T_+} \Tilde{E}(u(s); \al (T_+ -s)^{\frac{1}{2}}, A(T_+ -s)^{\frac{1}{2}}) = 0.
\end{align}
To see this, we just make use of \eqref{eq:en-ext-bu}. Since 
\begin{align}
  0 &\leq \Tilde{E}(u(s); \al (T_+ -s)^{\frac{1}{2}}, A(T_+ -s)^{\frac{1}{2}}) = \Tilde{E}(u(s)-u^*+u^*; \al (T_+ -s)^{\frac{1}{2}}, A(T_+ -s)^{\frac{1}{2}}) \\
  &\leq 2\Tilde{E}(u(s)-u^*; \al (T_+ -s)^{\frac{1}{2}}, A(T_+ -s)^{\frac{1}{2}})+ 2\Tilde{E}(u^*; \al (T_+ -s)^{\frac{1}{2}}, A(T_+ -s)^{\frac{1}{2}})\\
  &\leq 2\Tilde{E}(u(s)-u^*; \al (T_+ -s)^{\frac{1}{2}}, \infty)+ 2\Tilde{E}(u^*; \al (T_+ -s)^{\frac{1}{2}}, A(T_+ -s)^{\frac{1}{2}})
\end{align}
which tends to zero as $s\to T_+.$ Thus
\begin{align}\label{eqn:tilde-e-annulus}
\lim_{s \to T_+} \Tilde{E}(u(s); \al (T_+ -s)^{\frac{1}{2}}, A(T_+ -s)^{\frac{1}{2}}) = 0.
\end{align}
If we define the energy density \eqref{eqn:energy density} with just the gradient term, i.e.,
\EQ{
\Bar{E}= \int_0^\infty (\partial_r u)^2 r^{D-1}\dr
}
then since $\bar{E}(v)\leq \Tilde{E}(v)$ for any $v\in \mathcal{E}$ we get
\begin{align}
\lim_{s \to T_+} \bar{E}(u(s); \al (T_+ -s)^{\frac{1}{2}}, A(T_+ -s)^{\frac{1}{2}}) = 0.
\end{align}
Consequently, if we consider the nonlinear energy, then
\begin{align}
    E(u(t);\alp(T_+-t)^{\frac{1}{2}},A(T_+-t)^{\frac{1}{2}}) =\frac{1}{2} \bar{E}(u(t);\alp(T_+-t)^{\frac{1}{2}},A(T_+-t)^{\frac{1}{2}}) - \frac{1}{2^*} \int_{\alp(T_+-t)^{\frac{1}{2}}}^{A(T_+-t)^{\frac{1}{2}}} |u(t)|^{2^*}.
\end{align}
The first term in the RHS of the above equality tends to zero, so we only need to show that the second term also tends to zero. For this, we observe that by using a cut-off function and then the Sobolev inequality, we get
\begin{align}
  0\leq  \int_{\alp(T_+-t)^{\frac{1}{2}}}^{A(T_+-t)^{\frac{1}{2}}} |u|^{2^*} r^{D-1}\dr &\lesssim  \Tilde{E}\big(u(t);\alp (T_+-t)^{\frac{1}{2}}/2,2A(T_+-t)^{\frac{1}{2}}\big)^{2^*/2}\to 0
\end{align}
as $t\to T_+$, and therefore we get
\begin{align}
\lim_{t\to T_+} E(u(t); \al (T_+ -t)^{\frac{1}{2}}, A(T_+ -t)^{\frac{1}{2}}) = 0.
\end{align}
{Proof of \eqref{eq:energy-limit} and \eqref{eq:N-bubbles-bu}}. For any $\al>0$, using \eqref{eqn:finite-num-bubble} and \eqref{eq:radiation} we get
\begin{align}\label{eqn:e-d}
\lim_{t\to T_+} {E}(u(t)) &= L +{E}(u^*) \\
&= \lim_{t\to T_+} {E}(u(t);0,\alp(T_+-t)^{\frac{1}{2}}) + \lim_{t\to T_+} {E}(u(t);\alp(T_+-t)^{\frac{1}{2}},\infty)\\
&= \lim_{t\to T_+} {E}(u(t);0,\alp(T_+-t)^{\frac{1}{2}}) + \lim_{t\to T_+} {E}(u(t)-u^*+u^*;\alp(T_+-t)^{\frac{1}{2}},\infty)\\
&= \lim_{t\to T_+} {E}(u(t);0,\alp(T_+-t)^{\frac{1}{2}}) + \lim_{t\to T_+} {E}(u(t)-u^*;\alp(T_+-t)^{\frac{1}{2}},\infty) \\
&\quad + \lim_{t\to T_+} {E}(u^*;\alp(T_+-t)^{\frac{1}{2}},\infty)+\mathcal{O}\left(\lim_{t\to T_+} \| u(t)-u^*\|_{\E(\alp(T_+-t)^{\frac{1}{2}})}\right)\\
&= \lim_{t\to T_+} {E}(u(t);0,\alp(T_+-t)^{\frac{1}{2}})+ {E}(u^*).
\end{align}
Thus for any $\alp>0$ we have
\EQ{\label{eqn:L-on-parabolic-ball}
\lim_{t\to T_+} {E}(u(t);0,\alp(T_+-t)^{\frac{1}{2}}) = L.
}
We now aim to show that $L=NE(W)$ and that $N\geq 1.$ If $N=0$, then from the decomposition \eqref{eq:decomp-0} we deduce that $\|u(t_n)\|_{\calE(r\leq \alp(T_+-t_n)^{\frac{1}{2}})}\to 0$ as $n\to \infty$ which along with \eqref{eq:en-ext-bu}, implies that $0 \not \in \calS$ and therefore by $\veps$-regularity $u(t)$ is in $L^\infty$ for $r\in [0,\infty)$ up until $t=T_+$ contradicting the definition of the maximal blow up time (cf. Section 2.6 in \cite{collotdescription}). Thus $N\geq 1.$ To see that $L=NE(W)$ we observe that along the sequence $\{t_n\}$ in \eqref{eq:decomp-0} we have the following energy identity:
\begin{equation}
    \lim_{t_n\to T_+} E(u(t_n);0,\alp(T_+-t_n)^{\frac{1}{2}}) = N E(W)
\end{equation}
since $\lam_{n,N}\lesssim(T_+-t_n)^{\frac{1}{2}}$ if $N\geq 1.$ Therefore comparing this with \eqref{eqn:L-on-parabolic-ball} we deduce that $L=NE(W).$
\end{proof} 

\section{Sequential bubbling for global solutions}
The goal of this section is to establish sequential soliton resolution when $T_+=\infty.$
\begin{prop}[Sequential bubbling for global-in-time solutions]  \label{prop:seq-global}  
Let $u_0 \in \E$ and let $u(t)$ denote the solution to~\eqref{eqn:NLH} with initial data $u_0$ and bounded $\calE$-norm. Suppose that $T_+(u_0) = \infty$.  Then there exist $T_0\in (0,\infty)$, a unique integer $N \ge 0$, a sequence of times $t_n \to \infty$, signs $\vec\iota \in \{-1, 1\}^N$, a sequence of scales $\vec \lam_n \in (0, \infty)^N$, and a sequence of $g_n$ defined by 
\EQ{
u(t_n) = \sum_{j =1}^N \iota_j W_{\lam_{n,j}}  + g_n
}
with the following properties:
\begin{itemize} 
\item[(i)] the unique integer $N \ge 0 $ satisfies,  
\EQ{ \label{eq:energy-limit-global} 
\lim_{t \to \infty} E( u(t)) =  NE(W); 
}
\item[(ii)] for every $\alpha>0$, 
\EQ{ \label{eq:ext-energy-global} 
\lim_{t \to \infty} E( u(t); \al \sqrt{t}, \infty) = 0, 
} 
and there exist $T_0>0$ and a function $ \rho: [T_0, \infty) \to (0, \infty)$ such that 
\EQ{ \label{eq:ext-E-global} 
\lim_{t \to \infty} \Big( \frac{\rho(t)}{ \sqrt{t}} + \| u(t)\|_{\E(\rho(t))} \Big) = 0; 
}
\item[(iii)] the scales $\vec \lam_n$ and the sequence $g_n$ satisfy, 
\EQ{ \label{eq:global-seq} 
\lim_{ n \to \infty} \Big( \| g_n \|_{\E}^2 + \sum_{j =1}^N \Big( \frac{\lam_{n, j}}{\lam_{n, j+1}} \Big)^{\frac{D-2}{2}}  \Big)^{\frac{1}{2}}  = 0 
}
where here we adopt the convention that $\lam_{n, N+1} :=  t_n^{\frac{1}{2}}$. 
\end{itemize} 
\end{prop} 
\begin{proof} 
Let $u(t) \in \E$ be a solution of \eqref{eqn:NLH} defined globally in time. By~\eqref{eq:tension-L2} we can find a sequence $t_n \to \infty$ so that, $t_n^{\frac{1}{2}} \| \calT( u(t_n)) \|_{L^2}  \to 0 \mas n \to \infty. $ We can now apply Lemma~\ref{lem:loc-seq} which yields $N \ge 0$, $\vec \iota \in \{-1, 1\}^N, \vec \lam_n \in (0, \infty)^N$ such that after passing to a subsequence, we have 
 \begin{align}\label{eq:decomp-global-0} 
\lim_{n \to \infty} \Big( \| u(t_n) - \calW(\vec \iota, \vec \lam_n) \|_{\E( r \le A  \sqrt{t_n})}^2 + \sum_{j =1}^{N-1} \Big( \frac{ \lam_{n, j}}{\lam_{n, j+1}} \Big)^{\frac{D-2}{2}} \Big) = 0     
 \end{align}
for each $A>0$, and moreover that $\lam_{n, N} \lesssim  t_n^{\frac{1}{2}}$. Fix $\alpha>0$ and let $\varepsilon >0$. Then by~\eqref{eq:tension-L2} and $\|u_0\|_{\E}< \infty$  we can find $T_0 = T_0(\varepsilon) >0$ such that, 
\EQ{ \label{eq:T_0-choice} 
\Big(\int_{T_0}^\infty \int_0^\infty (\p_t u (t, r))^2 \, r^{D-1} \, \ud r \, \ud t\Big)^{\frac{1}{2}}  \le   \varepsilon.
}
Next, choose $T_1 \ge T_0$ so that 
\EQ{ \label{eq:T_1-choice} 
\|u(T_0)\|_{\E(\al \sqrt{T}/4, \infty)} \le \varepsilon
}
for all $T\geq T_1.$ Fixing any such $T$, we set 
\begin{align*}
 \phi(t, r) = \phi_T(r)  = 1 - \chi( 4r/ \al \sqrt{T}) \mfor t \in [T_0, T]. 
\end{align*}
Since $\frac{\ud }{\ud t} \phi(t, r) = 0$ for $t \in [T_0, T]$ it follows from~\eqref{eqn:modified energy ineq II} that
\begin{align}\label{eqn:local-monotonicity}
 &\int_0^{\infty} \Tilde{\mathbf{e}}(u(T)) \phi^2 - \int_0^{\infty} \Tilde{\mathbf{e}}(u(T_0))\phi^2  \\
&\le -2 \int_{T_0}^{T}\int_0^{\infty} (\partial_t u)^2  \phi^2 r^{D-1}\dr  \mathrm{d}t  + 4 \left(\int_{T_0}^{T}\int_0^{\infty} (\partial_t u)^2 \phi^2 (\partial_r \phi)^2 r^{D-1}\dr \mathrm{d}t\right)^{1/2}   \left(\int_{T_0}^{T}\int_0^{\infty} (\partial_r u)^2 r^{D-1}\dr \mathrm{d}t \right)^{1/2} \\
&\quad +2 \left(\int_{T_0}^{T}\int_0^{\infty} |u|^{2p} \phi^2 r^{D-1}\dr \mathrm{d}t\right)^{1/2}  \left(\int_{T_0}^{T}\int_0^{\infty} (\partial_t u)^2 \phi^2 r^{D-1} \dr \mathrm{d}t\right)^{1/2}  \\
&\quad +2 \left(\int_{T_0}^{T} \int_0^{\infty} \frac{|u|^2}{r^4} \phi^2 r^{D-1}\dr \mathrm{d}t\right)^{1/2}  \left(\int_{T_0}^{T} \int_0^{\infty}(\partial_t u)^2 \phi^2 r^{D-1}\dr \mathrm{d}t\right)^{1/2}\\
   &\lesssim  \left(\frac{T-T_0}{ {T}}\right)^{1/2}  \left(\int_{T_0}^{T}\int_0^{\infty} (\partial_t u)^2 r^{D-1}\ud r \ud t \right)^{1/2}+  \left(\int_{T_0}^{T} \int_{\frac{\alp \sqrt{T}}{4}}^{\infty} \frac{1}{r^{3}}\dr \mathrm{d}t \right)^{1/2}  \left(\int_{T_0}^{T} \int_{0}^\infty (\partial_t u)^2 \right)^{1/2}  \\
&\lesssim \left(\int_{T_0}^{\infty} \int_{0}^\infty (\partial_t u)^2 \right)^{1/2},
\end{align}
where we used $|\partial_r \phi|^2\lesssim T^{-1}$ and the radial Sobolev embedding \eqref{eqn:radial-sobolev} to control the integral involving $|u|^{2p}$. The constant in the above inequality depends on $\sup_{t\in [0,\infty)}\|u(t)\|_{\E}$ and $\alp$. So in particular we can make this small by choosing $T_0$ large enough such that $\Tilde{C}  \left(\int_{T_0}^{\infty} \int_{0}^\infty (\partial_t u)^2 \right)^{1/2}  \leq \frac{\varepsilon}{2}.$ Using the above together with~\eqref{eq:T_0-choice} and~\eqref{eq:T_1-choice}, we find that the Dirichlet energy can be made arbitrarily small
\EQ{\label{eq:energy-bound}
\Tilde{E}( u(T); \al \sqrt{T}, \infty) \le \varepsilon 
}
for all $T \ge T_1$ and $t\in [T_0,T]$. Thus by the localized coercivity lemma \ref{lem:localized-coercivity} we see that that ${E}( u(T); \al \sqrt{T}, \infty)\geq 0$ and moreover since
\begin{align}
    E( u(T); \al \sqrt{T}, \infty) \lesssim \Tilde{E}( u(T); \al \sqrt{T}, \infty) \le \varepsilon
\end{align}
we get ~\eqref{eq:ext-energy-global}. Since we proved ~\eqref{eq:ext-energy-global} for any $\alp>0$ there exists a curve $\rho$ such that \eqref{eq:ext-E-global} holds. Returning to the sequential decomposition we see from~\eqref{eq:decomp-global-0}, the fact that $\lam_{n, N} \lesssim t_n^{\frac{1}{2}},$ and from~\eqref{eq:ext-energy-global} that we must have 
\EQ{
\lim_{n \to \infty} \frac{\lam_{n, N}}{t_n^{\frac{1}{2}}}  = 0. 
}
Then,~\eqref{eq:global-seq} follows from the above,~\eqref{eq:ext-E-global} and~\eqref{eq:decomp-global-0}. Moreover, we see that $\lim_{n \to \infty} E( u(t_n))  =  N E(W)$ and the continuous limit~\eqref{eq:energy-limit-global} then follows from the fact that $E( u(t))$ is non-increasing. 
\end{proof}

\section{Decomposition of the solution and collision intervals}\label{sec:decomposition} 
For the remainder of the paper, we fix a solution $u(t)$ of \eqref{eqn:NLH}, with bounded energy (in $\calE$-norm) defined on the time interval $I_*=[0, T_+)$
where $T_+< \infty$ or $T_+=\infty$. By Propositions \ref{prop:seq-ftbu} and \ref{prop:seq-global} there exist a unique integer $N \ge0$, a unique function $u^*\in \calE$ and a sequence of times $t_n \to T_+$ so that $u(t_n)-u^*$ approaches an $N$-bubble as $n \to \infty$. When $T_+=\infty$, we simply set $u^*=0$ and when $T_+<\infty$ then $N\geq 1$. We define the localized distance to an $N$-bubble. 
\begin{defn}[Proximity to a multi-bubble]
\label{def:proximity}
For all $t \in I_*$, $\rho \in (0, \infty)$, and $K \in \{0, 1, \ldots, N\}$, we define
the \emph{localized multi-bubble proximity function} as
\begin{equation}
\bfd_K(t; \rho) := \inf_{\vec \iota, \vec\lam}\bigg( \|  u(t)  - u^*- \mathcal{W}( \vec\iota, \vec\lambda) \|_{\cE(\rho)}^2 + \sum_{j=K}^{N}\Big(\frac{ \lam_{j}}{\lam_{j+1}}\Big)^{\frac{D-2}{2}} \bigg)^{\frac{1}{2}},
\end{equation}
where $\vec\iota := (\iota_{K+1}, \ldots, \iota_N) \in \{-1, 1\}^{N-K}$, $\vec\lambda := (\lambda_{K+1}, \ldots, \lambda_N) \in (0, \infty)^{N-K}$, $\lambda_K := \rho$ and $\lambda_{N+1} := \sqrt{T_+-t}$ when $T_+<\infty$ and $\lam_{N+1}:=\sqrt{t}$ when $T_+=\infty.$ The \emph{multi-bubble proximity function} is defined by $\bfd(t) := \bfd_0(t; 0)$.
\end{defn}
\begin{rem} 
Observe that $\bfd_K(t; \rho)$ being small implies that $ u(t)-u^*$ is close to $N-K$ bubbles in the exterior region  $r \in (\rho, \infty)$. 
\end{rem} 
From Definition ~\ref{def:proximity} and Propositions ~\ref{prop:seq-ftbu} and ~\ref{prop:seq-global} we deduce that there exists a monotone sequence $t_n \to T_+$ such that
\begin{equation}
\label{eq:dtn-conv}
\lim_{n \to \infty} \bfd(t_n) = 0.
\end{equation}
Furthermore, proving Theorem~\ref{thm:main} would follow from showing that
\begin{equation}
\label{eq:dt-conv}
\lim_{t \to T_+} \bfd(t) = 0.
\end{equation}
Note that it suffices to prove Theorem~\ref{thm:main} in the case when $N\geq 1$. Indeed when $N=0$, $T_+=\infty$ by Proposition \ref{prop:seq-ftbu} and $\lim_{n\to \infty} \|u(t_n)\|_{\calE} = 0$ by Proposition \ref{prop:seq-global}, so the small data Cauchy theory from Lemma \ref{lem:lwp} implies $\lim_{t\to \infty}\|u(t)\|_{\calE}= 0$. Furthermore, as a direct consequence of \eqref{eq:radiation} and \eqref{eq:ext-E-global}, defining $\rho_N(t)=\rho(t)$ gives us the following lemma.
\begin{lem}
\label{lem:conv-rhoN}
There exist $T_0>0$ and a function $\rho_N: [T_0, T_+) \to (0, \infty)$ such that
\begin{equation}
\label{eq:conv-rhoN}
\lim_{t\to T_+}\bfd_N(t; \rho_N(t)) = 0.
\end{equation}
\end{lem}
\subsection{Collision intervals} \label{ssec:collision} To prove Theorem~\ref{thm:main} we analyze collision intervals. We recall their definition from \cite{lawrie-harmonic-map}.
\begin{defn}[Collision interval]
\label{def:collision}
Let $K \in \{0, 1, \ldots, N\}$. A compact time interval $[a, b] \subset I_*$ is a \emph{collision interval}
with parameters $0 < \varepsilon < \eta$ and $N - K$ exterior bubbles if
\begin{itemize}
\item $\bfd(a) \leq \varepsilon$ and $\bfd(b) \ge \eta$,
\item there exists a function $\rho_K: [a, b] \to (0, \infty)$ such that $\bfd_K(t; \rho_K(t)) \leq \varepsilon$
for all $t \in [a, b]$.
\end{itemize}
In this case, we write $[a, b] \in \calC_K(\varepsilon, \eta)$.
\end{defn}
\begin{defn}[Choice of $K$]
\label{def:K-choice}
We define $K$ as the \emph{smallest} nonnegative integer having the following property.
There exist $\eta > 0$, a decreasing sequence $\varepsilon_n \to 0$,
and sequences $a_n,b_n \to T_+$ such that $[a_n, b_n] \in \calC_K(\varepsilon_n, \eta)$ for all $n \in \{1, 2, \ldots\}$.
\end{defn}
Next, we observe that $K\geq 1$ since if \eqref{eq:dt-conv} is false, then at least one bubble must lose its shape.
\begin{lem}[Existence of $K \ge 1$]
\label{lem:K-exist}
If \eqref{eq:dt-conv} is false, then the number $K$ given by Definition \ref{def:K-choice} is well defined and $K \in \{1, \ldots, N\}$.
\end{lem}
\begin{proof}
The proof is similar to the proof of Lemma 5.6 in \cite{lawrie-wave}, but we recall it here for the reader's convenience. 

First, we will show that $K$ is well defined with $K\leq N$. It suffices to show that there exist $\eta >0$, a decreasing sequence $\varepsilon_n\to 0$, and sequences $(a_n), (b_n)$ such that $[a_n,b_n]\in \mathcal{C}_N(\varepsilon_n, \eta)$ for all $n\in \mathbb{N}.$ Since \eqref{eq:dt-conv} is false, there exist $\eta > 0$ and a monotone sequence $b_n \to T_+$ such that for all $n$, 
\begin{equation}
\bfd(b_n) \geq \eta.
\end{equation}
Next, we will find two sequences $(\varepsilon_n)$ and $(a_n)$ such that $[a_n, b_n] \in \calC_N(\varepsilon_n, \eta)$. By \eqref{eq:dtn-conv}, there exist $\varepsilon_n \to 0$ and a monotone sequence $a_n \to T_+$ such that $\bfd(a_n) \leq \varepsilon_n$. Taking a further subsequence of $b_n$, we may also assume that $a_n< b_n$. To verify the second point in the Definition \ref{def:collision}, we need to find a curve $\rho_N.$ To this end, we make use of Lemma \ref{lem:conv-rhoN} which yields the existence of a curve $\rho_N$ which we restrict to the time interval $[a_n,b_n].$ Then \eqref{eq:conv-rhoN} yields
\begin{equation}
\lim_{n\to\infty}\sup_{t\in[a_n, b_n]}\bfd_N(t; \rho_N(t)) = 0.
\end{equation}
Thus, up to changing $\varepsilon_n$, we see that all the conditions of Definition~\ref{def:collision} with $K=N$ are satisfied and hence $[a_n,b_n]\in \calC_N(\varepsilon_n,\eta).$

We now show that $K\geq 1.$ Suppose, to the contrary, that $K = 0$. By Definition~\ref{def:K-choice}, there exist $\eta>0$, a sequence $\veps_n\to 0$, and a sequence of intervals $[a_n,b_n]\in \calC_0(\veps_n,\eta)$ (with $a_n,b_n\to T_+)$. Without loss of generality, we may assume $\eta$ is small such that Lemma \ref{lem:bub-config} and Lemma \ref{lem:basic-trapping} are valid. By Definition~\ref{def:collision}, we can choose $c_n\in [a_n,b_n]$ and $\rho_n>0$ such that $\bfd(c_n)=\eta$ and $\bfd_0(c_n;\rho_n)\leq \veps_n.$ Since $\bfd_0\left(c_n; \rho_n\right) \leq \varepsilon_n$, there exist parameters $\rho_n \ll \lambda_{n, 1} \ll \cdots \ll \lambda_{n, N}$, and signs $\vec{\iota}_n$ such that
\EQ{\label{eqn:exterior-decomp}
\|{u}\left(c_n\right)-u^*-{\mathcal { W }}(\vec{\iota}_n, \vec{\lambda}_n)\|_{\mathcal{E}\left(\rho_n, \infty\right)}^2+\sum_{j=0}^N\left(\frac{\lambda_{n, j}}{\lambda_{n, j+1}}\right)^{\frac{D-2}{2}} \leq \varepsilon_n^2,
}
where $\lam_{n,N+1}=\sqrt{T_+-c_n}$ when $T_+<\infty$ and $\lam_{n,N+1}=\sqrt{c_n}$ when $T_+=\infty.$ The above display \eqref{eqn:exterior-decomp} together with Lemma \ref{lem:M-bub-energy} implies that
\EQ{
{E}({u}(c_n) ; \rho_n, \infty)&= {E}\left({u}\left(c_n\right)-u^*-{\mathcal { W }}(\vec{\iota}_n, \vec{\lambda}_n)+u^*+{\mathcal { W }}(\vec{\iota}_n, \vec{\lambda}_n) ;  \rho_n, \rho\left(c_n\right)\right) \\ & +{E}\left({u}\left(c_n\right)-u^*-{\mathcal { W }}(\vec{\iota}_n, \vec{\lambda}_n)+u^*+{\mathcal { W }}(\vec{\iota}_n, \vec{\lambda}_n) ; \rho(c_n), \infty\right) \\
&=N {E}(W)+{E}(u^*)+o_n(1) \text { as } n \rightarrow \infty.
}
Using \eqref{eq:energy-limit} when $T_+<\infty$,  \eqref{eq:energy-limit-global} when $T_+=\infty$, the above display implies that 
\EQ{
E(u(c_n);0, \rho_n) = o_n(1),\text{ as } n\to \infty.
}
Note that \eqref{eqn:exterior-decomp} also implies that
\EQ{\label{eqn:exterior-energy}
\tilde{E}({u}(c_n) ; \rho_n, \infty)=N \tilde{E}(W)+\tilde{E}(u^*)+o_n(1) \text{ as } n \rightarrow \infty
}
and since $\rho_n<2\rho_n\ll \lam_{n,1}$, $u^*=0$ when $T_+=\infty$, and $\|u^*\|_{\mathcal{E}(0,2\rho_n)}=o_n(1)$ when $T_+<\infty$ we get that
\EQ{\label{eqn:annulus}
\tilde{E}({u}(c_n) ; \rho_n, 2\rho_n)=o_n(1) \text{ as } n \rightarrow \infty.
}
Combining \eqref{eqn:exterior-energy} with $\bfd(c_n)=\eta$, yields
\EQ{
\tilde{E}({u}(c_n))=N \tilde{E}(W)+\tilde{E}(u^*)+O(\eta^2) \text{ as } n \rightarrow \infty.
}
Therefore, 
\EQ{\label{eqn:interior}
\tilde{E}({u}(c_n) ; 0,\rho_n) = O(\eta^2),\text{ as }n\to \infty,
}
Therefore, \eqref{eqn:annulus}, \eqref{eqn:interior} and Hardy's inequality imply that $v_n = u(c_n) \chi_{\rho_n}$ with $\chi\in C^\infty_c([0,2))$ satisfies $\|v_n\|_{\mathcal{E}}\lesssim \eta$. If we choose $\eta>0$ small enough (but fixed), then by Lemma \ref{lem:basic-trapping} we get 
\EQ{
\|u(c_n)\|^2_{\calE(0,\rho_n)}\leq \tilde{E}(v_n) \lesssim E(v_n) \lesssim E(u(c_n);0,\rho_n) + \|u(c_n)\|^2_{\calE(\rho_n,2\rho_n)} = o_n(1)
}
which implies that $\|u(c_n)\|_{\calE(0,\rho_n)}=o_n(1)$ as $n\to \infty.$ Therefore, we get
\EQ{
\eta=\bfd(c_n) &\leq \|{u}\left(c_n\right)-u^*-{\mathcal { W }}(\vec{\iota}_n, \vec{\lambda}_n)\|_{\calE} +\sum_{j=0}^N \left(\frac{\lam_{n,j}}{\lam_{n,j+1}}\right)^{\frac{D-2}{4}}\\
&\lesssim \|{u}\left(c_n\right)-u^*-{\mathcal { W }}(\vec{\iota}_n, \vec{\lambda}_n)\|_{\calE(0,\rho_n)}+\|{u}\left(c_n\right)-u^*-{\mathcal { W }}(\vec{\iota}_n, \vec{\lambda}_n)\|_{\calE(\rho_n,\infty)}+\veps_n\\
&\lesssim o_n(1) + \varepsilon_n\to 0
}
as $n\to \infty$ which contradicts that $\eta>0.$ Therefore $K\geq 1.$
\end{proof}

\begin{rem} \label{rem:collision} 
Without loss of generality, we assume that a collision interval $[a_n,b_n]\in \calC_K(\veps_n,\eta)$ satisfies $\bfd(a_n)=\veps_n$, $\bfd(b_n) = \eta$, and $\bfd(t) \in [\veps_n, \eta]$ for each $t \in [a_n, b_n]$. Furthermore, given parameters $\veps_n$ and $\eta$ we will often enlarge $\veps_n$ and shrink $\eta$ by choosing another sequence $\veps_n \le \ti \veps_n  \to 0$, and $0< \ti \eta \le \eta$, resulting in a small collision interval $[\ti a_n, \ti b_n]  \subset [a_n, b_n]$ and $[\ti a_n, \ti b_n]\subset~\calC_{K}(\ti \veps_n, \ti \eta)$. Since $a_n,b_n\to T_+$, passing to a further subsequence, we may assume that the collision intervals are disjoint. 
\end{rem} 
\subsection{Decomposition of the solution} 
In this section, we will perform a modulation analysis on the sequence of collision intervals. Amongst other things, in the Lemma below we will define two scales $\alp(t)$ and $\nu(t)$ that capture the bubbling region and separate the $K$ ``interior" bubbles with the $(N-K)$ ``exterior" bubbles respectively. This can be seen in equations \eqref{eqn:no-neck-nu} and \eqref{eqn:scale-separation} below.

\begin{lem}\label{lem:basic-modulation}
Let $K\geq 1$ be as in Lemma \ref{lem:K-exist}. There exist constants $C_0,\eta>0$, a sequence $\veps_n \rightarrow 0$, and sequences $a_n, b_n \rightarrow \infty$ satisfying Definition \ref{def:K-choice} with $\mathbf{d}\left(a_n\right)=\veps_n, \mathbf{d}\left(b_n\right)=\eta$ and $\mathbf{d}(t) \in\left[\veps_n, \eta\right]$ for all $t \in\left[a_n, b_n\right]$ such that the following is true. There exist signs $\vec{\iota} \in\{-1,1\}^N$, a function $\vec{\lambda}=\left(\lambda_1, \ldots, \lambda_N\right) \in C^1\left(\cup_{n \in \mathbb{N}}\left[a_n, b_n\right] ;(0, \infty)^N\right)$, sequences $\alpha_n \rightarrow 0$ and $\nu_n \rightarrow 0$, such that the functions,
\begin{align}
& \nu: \cup_{n \in \mathbb{N}}\left[a_n, b_n\right] \rightarrow(0, \infty), \quad \nu(t):=\nu_n \lambda_{K+1}(t), \quad \text { for } \quad t \in\left[a_n, b_n\right], \label{defn:nu}\\
& \alpha: \cup_{n \in \mathbb{N}}\left[a_n, b_n\right] \rightarrow(0, \infty), \quad \alpha(t):=\left\{\begin{array}{l}
\alpha_n \sqrt{T_{+}-t} \text {, if } T_{+}<\infty \\
\alpha_n \sqrt{t} \text {, if } T_{+}=\infty
\end{array} \quad, \text { for } t \in\left[a_n, b_n\right],\right. \label{defn:alpha}\\
& u^*(t):=\left\{\begin{array}{l}
\left(1-\chi_{\alpha(t)}\right)u(t) \text {, if } T_{+}<\infty \\
\quad 0 \quad \quad \quad \quad \quad   \text {, if } T_{+}=\infty 
\end{array}\right. \text {, for } t \in\left[a_n, b_n\right] \label{defn:u^*}, 
\end{align}
and
\begin{align}
g: \cup_{n \in \mathbb{N}}\left[a_n, b_n\right] \rightarrow \mathcal{E};\quad  g(t) &= u(t)-u^*(t) -{\mathcal { W }}(\vec{\iota}, \vec{\lambda}(t)),\label{eq:g-def}
\end{align}
satisfy the following orthogonality conditions
\EQ{
\left\langle\mathcal{Z}_{\underline{\lambda_j(t)}} \mid g(t)\right\rangle & =0,  \quad \forall t \in\left[a_n, b_n\right],\quad \forall n\in \N \label{eq:g-ortho};
}
and estimates,
\begin{align}
\lim _{n \rightarrow \infty} \sup _{t \in\left[a_n, b_n\right]}&\left(\frac{\nu(t)}{\lambda_{K+1}(t)}+\sum_{j=K+1}^{N-1} \frac{\lambda_j(t)}{\lambda_{j+1}(t)}+\frac{\lambda_N(t)}{\alpha(t)}+\tilde{E}\left(u(t) ; \frac{1}{4} \nu(t), 4 \nu(t)\right)\right)=0, \label{eqn:no-neck-nu}\\
C_0^{-1} \mathbf{d}(t)&\leq\|{g}(t)\|_{\mathcal{E}}+\sum_{j=1}^{N-1}\left(\frac{\lambda_j(t)}{\lambda_{j+1}(t)}\right)^{\frac{D-2}{4}} \leq C_0 \mathbf{d}(t)\label{eq:d-g-lam}, \\
\abs{\lam_j'(t)} &\le \frac{C_0}{\lambda_j(t)} \bfd(t),\text{ for all }j\in \{1,\ldots,N\} \label{eq:lambda'-bound}
\end{align}
for all $t \in\left[a_n, b_n\right]$ and all $n \in \mathbb{N}$. Furthermore, for any sequence $s_n \in\left[a_n, b_n\right]$ 
\EQ{
\lim _{n \rightarrow \infty} {E}\left(u\left(s_n\right) ; \nu(s_n), \infty\right)=(N-K) {E}(Q)+{E}\left(u^*\right)\label{eqn:scale-separation}
}
and, if we denote $g^{\mathrm{ext}}_n=u(s_n)-u^*(s_n)-\mathcal{W}(\iota_{K+1}, \ldots, \iota_N,\lambda_{K+1}(s_n), \ldots, \lambda_N(s_n))$, then 
\EQ{\label{eqn:exterior-bubble}
\lim_{n \rightarrow \infty}\left( \| g^{\mathrm{ext}}_n \|_{\mathcal{E}(r \geq \nu(s_n))} +\sum_{j=K+1}^N\left(\frac{\lambda_j(s_n)}{\lambda_{j+1}(s_n)}\right)^{\frac{D-2}{4}}\right)=0. 
}
\end{lem}
\begin{proof} 
We closely follow the proof of Lemma 5.9 in \cite{lawrie-harmonic-map} and restrict ourselves to the case when $T_+<\infty$ since the case when $T_+=\infty$ is similar and less cumbersome due to the fact that $u^*=0.$ 

First by Definition \ref{def:K-choice} and Lemma \ref{lem:K-exist}, we can find parameters $a_n, b_n, \veps_n, \eta$, and $K \in\{1, \ldots, N\}$. As mentioned in Remark \ref{rem:collision}, we will shrink $\eta$ and enlarge $\veps_n$, but by abuse of notation we will denote the resulting sub-intervals again by $[a_n,b_n]$.

Next, we define the scales $\alp(t)$ and $\nu(t)$ as in \eqref{defn:alpha} and \eqref{defn:nu} respectively. To this end, using Definition \ref{def:proximity}, for all $n\in \N$, there exist scales $\rho_K(t) \ll \mu_{K+1}(t) \ll \cdots \ll \mu_N(t) \ll\left(T_{+}-t\right)^{\frac{1}{2}}$ and signs $\vec{\sigma}(t) \in\{-1,1\}^{N-K}$ for $t \in\left[a_n, b_n\right]$, such that the error function $\tilde{g}_{\rho_K}(t)$ defined on the region $r \in\left(\rho_K(t), \infty\right)$ by the identity
\EQ{
u(t)-u^*=\mathcal{W}\left(\vec{\sigma}(t), \vec{\mu}(t)\right)+\tilde{g}_{\rho_K}(t)
}
satisfies 
\EQ{\label{eqn:d(rho_K) small}
\mathbf{d}\left(t ; \rho_K(t)\right) \simeq\left\|\tilde{g}_{\rho_K}(t)\right\|_{\mathcal{E}\left(\rho_K(t), \infty\right)}^2+\sum_{j=K}^N\left(\frac{\mu_j(t)}{\mu_{j+1}(t)}\right)^{\frac{D-2}{2}} \lesssim \veps_n^2
}
where by convention $\mu_K(t):=\rho_K(t)$ and $\mu_{N+1}(t):=\left(T_{+}-t\right)^{\frac{1}{2}}$. Since we know that 
\begin{align*}
\lim_{n \rightarrow \infty}\sup_{t \in\left[a_n, b_n\right]}\mathbf{d}_K\left(t ; \rho_K(t)\right)=0,    
\end{align*}
and 
\EQ{
\lim _{n \rightarrow \infty} \sup _{t \in\left[a_n, b_n\right]} \|\mathcal{W}\left(\vec{\sigma}(t), \vec{\mu}(t)\right)\|_{\mathcal{E}(\nu_{n, 1} {\mu}_{K+1}(t), \nu_{n, 2} {\mu}_{K+1}(t))} =0,
}
for any two sequences $\nu_{n,1}\ll \nu_{n,2}\ll 1$, we can choose a sequence $\nu_n\to 0$ such that for any $A>1$ we have
\EQ{\label{defn:seq-nu_n}
\rho_K(t) \leq \nu_n \mu_{K+1}(t), \text{ and } \lim _{n \rightarrow \infty} \sup _{t \in\left[a_n, b_n\right]} \|u(t)-u^*\|_{\mathcal{E}(A^{-1}\nu_n \mu_{K+1}(t), A \nu_n \mu_{K+1}(t))}=0.
}
Next, using $\rho(t)$ as in \eqref{eq:radiation} and \eqref{eqn:d(rho_K) small} we can choose a sequence $\alp_n\to 0$ such that 
\EQ{\label{def:alpha_n-seq}
\lim _{n \rightarrow \infty} \sup _{t \in\left[a_n, b_n\right]}\left(\frac{\mu_N(t)}{\alpha_n\left(T_{+}-t\right)^{\frac{1}{2}}}+\frac{\rho(t)}{\alpha_n\left(T_{+}-t\right)^{\frac{1}{2}}}\right)=0.
}
Thus, we define $\alpha(t):=\alpha_n\left(T_{+}-t\right)^{\frac{1}{2}}$ for $t \in\left[a_n, b_n\right]$. If $K=N$ we may assume that $\alpha_n \geq \nu_n$. 
\newline 
Next, we define $u^*(t)$. Let
\EQ{\label{eqn:u^*(t)-defn}
u^*(t):=\left(1-\chi_{\alpha(t)}\right)u(t),
}
then from \eqref{eq:radiation} and $\lim _{t \rightarrow T_{+}} \|u^*\|_{\cE(0,\gamma(t))}=0$ for any $\gamma(t) \rightarrow 0$ as $t \rightarrow T_{+}$, we deduce that 
\EQ{\label{eqn:u^*(t)-u^*}
\lim _{t \rightarrow T_{+}}\left\|u^*(t)-u^*\right\|_{\mathcal{E}}=0.
}
Thus, using $u(t)-u^*(t)=\chi_{\alp(t)}u(t)$, \eqref{eq:N-bubbles-bu} and \eqref{eq:en-annulus} we get
\EQ{\label{eqn:u^*-energy-identity}
\lim _{n \rightarrow \infty} \sup _{t \in\left[a_n, b_n\right]}\left|E\left(u(t)-u^*(t)\right)-N E(Q)\right|=0.
}
Then, by definition $\bfd(t)$ we can find some signs $\vec{\tilde{\iota}}(t)\in \{-1,1\}^N$ and scale $\vec{\tilde{\lam}}(t)\in (0,\infty)^N$ such that the decomposition
\EQ{\label{def:tilde-g}
\tilde{g}(t):=u(t)-u^*-\calW(\vec{\tilde{\iota}}(t),\vec{\tilde{\lam}}(t))
}
satisfies 
\EQ{\label{eqn:error-tilde-g}
\mathbf{d}(t) \leq\|\widetilde{g}(t)\|_{\mathcal{E}}+\sum_{j=1}^N\left(\frac{\widetilde{\lambda}_j(t)}{\widetilde{\lambda}_{j+1}(t)}\right)^{\frac{D-2}{2}} \leq 2 \mathbf{d}(t) \leq 2 \eta
}
with the convention that $\tilde{\lam}_{N+1}(t)=(T_+-t)^{\frac{1}{2}}.$ Using, \eqref{eqn:u^*(t)-u^*}, \eqref{def:tilde-g} and \eqref{eqn:error-tilde-g} we get that
\EQ{\label{eqn:dist-upper-bound}
\operatorname{dist}_N(u(t)-u^*(t)) \leq C_0 \bfd(t) + o_n(1) \leq 2C_0 \eta,
}
where $\operatorname{dist}_N$ was defined in \eqref{def-d}. Then as explained in Remark \ref{rem:collision}, we shrink $\eta>0$ small enough such that we can apply Lemma \ref{lem:mod-static} to the function $u(t)-u^*(t)$, yielding the existence of $\vec{\lam}(t)\in (0,\infty)^N$ defined on $\cup_{n\in \N}[a_n,b_n]$ and signs $\vec{\iota}\in \{-1,1\}^N$ (which can be taken to be independent of time $t \in [a_n,b_n]$ using the continuity of the flow and independent of $n$ by passing to a subsequence of $[a_n,b_n]$) and $g(t)$ such that 
\EQ{
u(t)-u^*(t)=\calW(\vec{\iota},\vec{\lam})+g(t),\quad \left\langle\mathcal{Z}_{\underline{\lambda_j(t)}} \mid g(t)\right\rangle & =0,  \quad \forall t \in\left[a_n, b_n\right],\quad \forall n\in \N,
}
and
\EQ{
\operatorname{dist}_N(u(t)-u^*(t))\leq \|g(t)\|_{\cE} + \sum_{j=1}^{N-1}\left(\frac{\lam_j(t)}{\lam_{j+1}(t)}\right)^{\frac{D-2}{4}}\leq C_0 \operatorname{dist}_N(u(t)-u^*(t)).
}
Then using \eqref{eqn:u^*(t)-u^*} and \eqref{eqn:dist-upper-bound} we get
\EQ{
\mathbf{d}(t)-\xi_{1, n} \leq\|g(t)\|_{\mathcal{E}}+\sum_{j=1}^N\left(\frac{\lambda_j(t)}{\lambda_{j+1}(t)}\right)^{\frac{D-2}{4}} \leq C_0 \mathbf{d}(t)+\xi_{1, n}
}
for some sequence $\xi_{1,n}\to 0$ as $n\to \infty$. Thus enlarging $\veps_n$ so that $\veps_n\geq 2\xi_{1,n}$ for all $n\in \N$ as in Remark \ref{rem:collision} we deduce \eqref{eq:d-g-lam}.

Next, we establish \eqref{eqn:no-neck-nu}. As a first step, we will show that for each $j=1, \ldots, N$,
\EQ{\label{eqn:tilde-nu-lam_j}
\lim _{n \rightarrow \infty} \sup _{t \in\left[a_n, b_n\right]}\left(\frac{\widetilde{\nu}(t)}{\lambda_j(t)}+\frac{\lambda_j(t)}{\widetilde{\nu}(t)}\right)=0,
}
where $\widetilde{\nu}(t):=\nu_n \mu_{K+1}(t)$. If not, then there exist a constant $C>0$, an index $j \in\{1, \ldots, N\}$, a subsequence of the interval $\left[a_n, b_n\right]$ and a sequence $s_n \in\left[a_n, b_n\right]$ such that
\EQ{
C^{-1} \widetilde{\nu}\left(s_n\right) \leq \lambda_j\left(s_n\right) \leq C \widetilde{\nu}\left(s_n\right)
}
By \eqref{eq:d-g-lam}, for all $\eta>0$ small enough there exist $\delta=\delta(\eta)$, and $R=R(\eta)>0$ such that for all $n\in \N$ 
\EQ{
\delta \leq \|u\left(s_n\right)-u^*\left(s_n\right)\|_{\cE(R^{-1} \lambda_j\left(s_n\right), R \lambda_j\left(s_n\right))}  \leq \|u\left(s_n\right)-u^*\left(s_n\right)\|_{\cE(C^{-1} R^{-1} \widetilde{\nu}\left(s_n\right), R C \widetilde{\nu}\left(s_n\right))}
}
which contradicts the second equation in \eqref{defn:seq-nu_n}. By \eqref{defn:seq-nu_n}, if we define the function
\EQ{
w(t)= (1-\chi_{\tilde{\nu}(t)}) (u(t)-u^*(t))
}
then 
\EQ{\label{eqn:index-j_0}
\|w(t)-\cW(\vec{\sigma}(t),\vec{\mu}(t))\|_{\cE}^2+\sum_{j=K+1}^{N-1}\left(\frac{\mu_{j}(t)}{\mu_{j+1}(t)}\right)^{\frac{D-2}{2}} = o_n(1).
}
On the other hand, by \eqref{eqn:tilde-nu-lam_j} we can find an index $j_0 \in\{1, \ldots, N-1\}$ such that 
\EQ{
\left\|w(t)-\mathcal{W}\left(\iota_{j_0}, \ldots, \iota_N, \lambda_{j_0}(t), \ldots, \lambda_N(t)\right)\right\|_{\mathcal{E}}^2+\sum_{j=j_0}^{N-1}\left(\frac{\lambda_j(t)}{\lambda_{j+1}(t)}\right)^{\frac{D-2}{2}} \leq C_0 \eta.
}
Then Lemma \ref{lem:bub-config} yields $j_0=K+1, \vec{\sigma}(t)=\left\{\iota_{K+1}, \ldots, \iota_K\right\}$, and after shrinking $\eta>0$ we get that
\EQ{
\sup _{t \in\left[a_n, b_n\right]}\left|\frac{\lambda_j(t)}{\mu_j(t)}-1\right| \leq \frac{1}{4}.
}
Thus, defining $\nu(t):=\nu_n \lambda_{K+1}(t)$ we deduce \eqref{eqn:no-neck-nu} from \eqref{eqn:d(rho_K) small}, \eqref{defn:seq-nu_n} and \eqref{def:alpha_n-seq}.

Next, we prove \eqref{eqn:scale-separation}. Let $s_n \in\left[a_n, b_n\right]$ and $R_n$ such that $\nu\left(s_n\right) \leq R_n \ll \lambda_{K+1}\left(s_n\right)$. If $K<N$ then $R_n \ll \alpha\left(s_n\right)$, thus, using \eqref{eqn:index-j_0} and \eqref{def:alpha_n-seq} we get
\EQ{
E(u(s_n) ; R_n, \alpha(s_n)) \rightarrow(N-K) E(Q) \text { as } n \rightarrow \infty.
}
Since by \eqref{def:alpha_n-seq}, \eqref{eqn:u^*(t)-defn} and \eqref{eqn:u^*(t)-u^*},
\EQ{
E(u(s_n) ; \alpha(s_n), \infty) \rightarrow E(u^*) \text { as } n \rightarrow \infty,
}
combining this with the previous display yields \eqref{eqn:scale-separation}. If $K=N$ then $E(u(s_n) ; R_n, \infty) \rightarrow E(u^*)$. Finally, the equation \eqref{eqn:exterior-bubble} follows from \eqref{eqn:index-j_0}.

Now we prove the dynamical estimate \eqref{eq:lambda'-bound}. Differentiating in time the orthogonality conditions~\eqref{eq:g-ortho} yields, for each $j = 1, \dots, K$, the identity, 
\begin{align}\label{eq:diff-ortho}
0 & =-\frac{\lambda_j^{\prime}}{\lambda_j}\left\langle\underline{\Lambda} \mathcal{Z}_{\underline{\lambda_j}} \mid g\right\rangle+\left\langle\mathcal{Z}_{\underline{\lambda_j}} \mid \partial_t g\right\rangle.
\end{align}
Next, differentiating in time the expression for $g(t)$ in~\eqref{eq:g-def}
\EQ{\label{eq:g-eq} 
\p_t g &= \chi_\alpha \partial_t u(t) -u(t) \frac{\alpha^{\prime}(t)}{\alpha(t)}\left(r \partial_r \chi\right)(\cdot / \alpha(t))+\sum_{j=1}^N \iota_j \lam_j' \Lam W_{\U{\lam_j}}  \\
& =  \chi_\alpha \left(\De u+  f(u)\right) -u(t) \frac{\alpha^{\prime}(t)}{\alpha(t)}\left(r \partial_r \chi\right)(\cdot / \alpha(t)) + \sum_{j=1}^N \iota_j \lam_j' \Lam W_{\U{\lam_j}}\\
&= \Delta\left(\chi_\alpha u\right)+ f\left(\chi_\alpha u\right)+\sum_{j=1}^N \iota_j \lambda_j^{\prime} \Lambda W_{\underline{\lambda_j}} \\ 
&\quad -u \Delta \chi_\alpha-2 \partial_r u \partial_r \chi_\alpha-u(t) \frac{\alpha^{\prime}(t)}{\alpha(t)}\left(r \partial_r \chi\right)(\cdot / \alpha(t)) +\left(f(u) \chi_\alpha-f\left(\chi_\alpha u\right)\right)\\
& = \De g + f'(\calW(\vec{\iota},\vec{\lambda}))g + f(\calW(\vec{\iota},\vec{\lambda}))- \sum_{j=1}^{N}\iota_j f(W_{{\lam_j}}) \\
&\quad  +  f(\calW(\vec{\iota},\vec{\lambda})+g)- f(\calW(\vec{\iota},\vec{\lambda}))-f'(\calW(\vec{\iota},\vec{\lambda}))g \\
&\quad +\left(-u\Delta \chi_\alpha-2\partial_r u \partial_r \chi_\alpha -u(t) \frac{\alpha^{\prime}(t)}{\alpha(t)}\left(r \partial_r \chi\right)(\cdot / \alpha(t))+f(u)\chi_\alpha-f(\chi_\alpha u)\right)+\sum_{j=1}^N \iota_j \lam_j' \Lam W_{\U{\lam_j}}\\
&= -\mathcal{L}_{\calW} g + f_{\bfi}(\vec \iota, \vec \lam)+ f_{\bfq}(\vec \iota, \vec \lam, g) + \phi(u,\alpha)+ \sum_{j=1}^N \iota_j \lam_j' \Lam W_{\U{\lam_j}},
}
where
\EQ{
\calL_{\calW} &:= -\Delta -f'(\calW(\vec{\iota},\vec{\lambda})),\quad f_{\bfi}(\vec \iota, \vec \lam) := f\big( \mathcal{W}( \vec \iota, \vec \lam)\big) - \sum_{j =1}^{N} \iota_j  f(W_{\lam_{j}}) \\
f_{\bfq}(\vec \iota, \vec \lam, g) &:= f\big( \mathcal{W}( \vec \iota, \vec \lam) + g \big)  - f\big(\mathcal{W}( \vec \iota, \vec \lam) \big) -  f'\big( \mathcal{W}( \vec \iota, \vec \lam)\big) g\\
\phi(u,\alpha)&:=-u\Delta \chi_\alpha-2\partial_r u \partial_r \chi_\alpha -u(t) \frac{\alpha^{\prime}(t)}{\alpha(t)}\left(r \partial_r \chi\right)(\cdot / \alpha(t))+f(u)\chi_\alpha-f(\chi_\alpha u).
}
The subscript  $\bfi$ above stands for ``interaction'' and $\bfq$ stands for ``quadratic.''  For each $j \in \{1, \dots, K\}$ we pair~\eqref{eq:g-eq} with $\calZ_{\U{\lam_j}}$ and use~\eqref{eq:diff-ortho} to obtain the following system 
\begin{multline} 
\iota_j \lam_j' \Big( \La \Lam W \mid \calZ \Ra - \frac{\iota_j}{\lam_j} \La \U{\Lam}\calZ_{\U{\lam_j}} \mid g\Ra \Big) + \sum_{i \neq j} \iota_i  \lam_i'\La \Lam W_{\U{\lam_i}} \mid \calZ_{\U{\lam_j}} \Ra \\
= \La \calL_{\calW} g \mid \calZ_{\U{\lam_j}} \Ra - \La f_{\bfi}( \vec \iota, \vec \lam)\mid \calZ_{\U{\lam_j}} \Ra - \La f_{\bfq}(\vec \iota, \vec \lam, g)\mid \calZ_{\U{\lam_j}} \Ra - \La \phi(u,\alpha) \mid \calZ_{\U{\lam_j}}\Ra. 
\end{multline} 
The above system is diagonally dominant for all sufficiently small  $\eta_0>0$ due to \eqref{eq:ZQ} and \eqref{eq:d-g-lam}. Hence, it is invertible. Next, we estimate each term
\EQ{\label{eq:fi}
&\Big|  \La \calL_{\calW} g \mid \calZ_{\U{\lam_j}} \Ra  \Big| \lesssim \frac{1}{\lam_j}  \|g \|_{\E},\\
&\Big| \La f_{\bfi}(\vec \iota, \vec \lam) \mid \calZ_{\U{\lam_j}}  \Ra\Big|  \lesssim \frac{1}{\lam_j} \Big( \Big(\frac{\lam_{j}}{\lam_{j+1}} \Big)^{\frac{D-2}{2}} + \Big(\frac{\lam_{j-1}}{\lam_{j}} \Big)^{\frac{D-2}{2}} \Big) \lesssim \frac{1}{\lam_j} (\bfd(t)^2 + o_n(1)), \\
&\Big| \La f_{\bfq}(\vec \iota, \vec \lam, g) \mid \calZ_{\U{\lam_j}}   \Ra \Big| \lesssim \frac{1}{\lam_j} (\bfd(t) + o_n(1)), \quad \Big| \La \phi(u,\alpha) \mid \calZ_{\U{\lam_j}}\Ra \Big| \lesssim \frac{1}{\lam_j}  o_n(1),
}
where the first inequality follows by integration by parts and Cauchy-Schwarz, the second inequality follows from a computation analogous to Lemma 2.21 in \cite{lawrie-wave}, the third inequality follows from a Taylor expansion, and the fourth inequality follows from \eqref{eqn:no-neck-nu}. Thus, collecting all the above estimates, we get
\EQ{
\abs{\lam_j'} \lesssim \frac{1}{\lam_j} \big( \bfd(t) + \zeta_{3, n} \big) 
}
for some sequence $\zeta_{3, n} \to 0$ as $n \to \infty$. Then~\eqref{eq:lambda'-bound} follows by enlarging $\veps_n$. 
\newline
\end{proof}



\begin{lem}\label{lem:delta-to-d}
There exists a constant $\eta_0>0$ having the following property. Let $t_n \in\left[a_n, b_n\right]$ and let $\mu_n$ be a positive sequence satisfying the conditions:
\begin{enumerate}
    \item $\lim _{n \rightarrow \infty} \frac{\mu_n}{\lam_{K+1}\left(t_n\right)}=0,$
    \item $\mu_n \geq \nu\left(t_n\right)$ or $\left\|{u}\left(t_n\right)\right\|_{\mathcal{E}\left(\mu_n, \nu\left(t_n\right)\right)} \leq \eta_0$, and
    \item $\lim _{n \rightarrow \infty} \delta_{\mu_n}\left(u(t_n)\right)=0$, where $\delta_{\mu_n}$ is defined in \eqref{defn:loc-dist}.
\end{enumerate}
Then $\lim _{n \rightarrow \infty} \mathbf{d}\left(t_n\right)=0$.
\end{lem}
\begin{proof} Let $R_n$ be a sequence such that $\mu_n \ll R_n \ll \lam_{K+1}\left(t_n\right)$. Without loss of generality, we can assume $R_n \geq \nu\left(t_n\right)$, since it suffices to replace $R_n$ by $\nu\left(t_n\right)$ for all $n$ such that $R_n<\nu\left(t_n\right)$. By the definition of the localized distance \eqref{defn:loc-dist}, let $M_n, \vec{\iota}_n, \vec{\lambda}_n$ be parameters such that
\EQ{\label{eqn:loc-decomp}
\|u\left(t_n\right)-\mathcal{W}(\vec{\iota}_n, \vec{\lambda}_n)\|_{\E\left(r \leq \mu_n\right)}^2+\sum_{j=1}^{M_n}\left(\frac{\lambda_{n, j}}{\lambda_{n, j+1}}\right)^{\frac{D-2}{2}} \rightarrow 0
}
as $n\to \infty$. Set
$$
\begin{aligned}
{u}_n^{(i)} & :=\chi_{\frac{1}{2} \mu_n} {u}\left(t_n\right), \quad {u}_n^{(o)}  :=\left(1-\chi_{R_n}\right) {u}\left(t_n\right), \quad {u}_n^{(m)} :={u}\left(t_n\right)-{u}_n^{(i)}-{u}_n^{(o)} .
\end{aligned}
$$
Observe that if $\mu_n$ is a positive sequence such that $\lim _{n \rightarrow \infty} {\delta}_{\mu_n}\left(t_n\right)=0$, then
\EQ{\label{eqn:mu_n-annulus}
\lim _{n \rightarrow \infty}\|{u}\left(t_n\right)\|_{\mathcal{E}\left(\frac{1}{2} \mu_n, \mu_n\right)}=0.
}
Combining this with the localized decomposition \eqref{eqn:loc-decomp}, we have
\EQ{\label{eqn:inner-energy}
\lim _{n \rightarrow \infty}\|{u}_n^{(i)}-\mathcal{W}(\vec{\iota}_n, \vec{\lambda}_n)\|_{\mathcal{E}}=0 .
}
Furthermore, as $\bfd_K(t_n;\nu(t_n))\to 0$ from Lemma \ref{lem:basic-modulation}, $\nu\left(t_n\right) \leq R_n \ll \lam_{K+1}\left(t_n\right)$ implies
\EQ{\label{eqn:Rn-annulus}
\lim _{n \rightarrow \infty}\left\|{u}\left(t_n\right)\right\|_{\mathcal{E}\left(R_n, 2 R_n\right)}=0.
}
Thus using the second assumption along with \eqref{eqn:mu_n-annulus} and \eqref{eqn:Rn-annulus} for $n$ large enough we have
\EQ{
\|{u}_n^{(m)}\|_{\mathcal{E}}\leq 2 \eta_0, 
}
which implies from Lemma \ref{lem:basic-trapping} that $0 \leq E({u}_n^{(m)}) \lesssim \eta_0^2$. Using \eqref{eqn:mu_n-annulus} and \eqref{eqn:Rn-annulus}, we get
\EQ{\label{eqn:middle-to-outer}
\limsup _{n \rightarrow \infty}|E({u}(t_n))-E({u}_n^{(i)})-E({u}_n^{(m)})-E({u}_n^{(o)})|=0 .
}
By Lemma \ref{lem:basic-modulation} we have 
\EQ{
\lim_{n\to \infty} E(u_n^{(o)}) = (N-K) E(W) + E(u^*).
}
Combining this with $0 \leq E({u}_n^{(m)}) \lesssim \eta_0^2$ we deduce that $M_n=K$ and $\lim_{n\to \infty} E(u_n^{(m)})=0.$ By Lemma \ref{lem:basic-trapping}, $E(u_n^{(m)})\simeq \|u_n^{(m)}\|^2_{\calE}$ and thus, we get that 
\EQ{\label{eqn:middle-energy}
\lim _{n \rightarrow \infty}\|{u}_n^{(m)}\|_{\mathcal{E}}=0.
}
Thus, combining the proximity to $K$ bubbles in the inner region \eqref{eqn:inner-energy}, no energy lost in the middle region \eqref{eqn:middle-energy}, convergence to $N-K$ bubbles in the outer region since $\bfd_K(t_n;\nu(t_n))=o_n(1)$ with no loss of energy in the neck region \eqref{eqn:middle-to-outer} yields the convergence of the distance $\bfd(t_n)\to 0$ as $n\to \infty$ with parameters $\vec{\lam}_n=(\lam_{n,1},\ldots,\lam_{n,K},\lam_{K+1}(t_n),\ldots,\lam_{N}(t_n))$ and $\vec{\iota}_n=(\iota_{n,1},\ldots,\iota_{n,K},\iota_{K+1},\ldots,\iota_{N})$ with the last $(N-K)$ parameters coming from Lemma \ref{lem:basic-modulation}.
\end{proof}

\section{Conclusion}
We are now in a position to conclude the proof of Theorem~\ref{thm:main}. The idea is to first observe that the length of the collision interval is related to the scale of the $K$-th bubble. This follows from the dynamical estimate \eqref{eq:lambda'-bound} and the fact that $K$ is minimal as in Definition \ref{def:K-choice}. As a consequence, one can identify a suitable sub-interval of the collision interval in which the scale $\lam_K$ does not change much. Finally, combining the previous observation and the fact that the tension is finite, one can deduce a contradiction, which in particular implies that \eqref{eq:dt-conv} and thus Theorem~\ref{thm:main} holds.

For convenience, we assume that whenever $[a_n, b_n] \in \calC_K(\veps_n, \eta)$ then $\bfd(a_n) = \veps_n$, $\bfd(b_n) = \eta$ and $\bfd(t) \in [\veps_n, \eta]$ for all $t \in [a_n, b_n]$. This can always be done by Remark~\ref{rem:collision}. 
\begin{lem} \label{lem:collision-duration} 
For any $\eta>0$, there exist $\veps \in (0, \eta)$, $C_u>0$ and $n_0=n_0(u)\in \N$ with the following property. If $[c, d] \subset [a_n, b_n]$ for some $n\geq n_0$, $\bfd(c)  \le \veps$ and $\bfd(d) \ge \eta$, then, 
\EQ{
  (d- c)^{\frac{1}{2}} \ge C_u^{-1}  \lam_K(c).
  }
\end{lem} 

\begin{proof} If not, then there exist $\eta>0$, sequences $\veps_n \to 0$, $[c_n, d_n] \subset [a_n, b_n]$, and $C_n \to \infty$ so that $\bfd(c_n)  \le \veps_n$, $\bfd(d_n) \ge \eta$ and 
\EQ{ \label{eq:short-time} 
(d_n - c_n)^{\frac{1}{2}} \le C_n^{-1} \lam_K(c_n).
} 
We will show that $[c_n,d_n]\in \calC_{K-1}(\ti\veps_n,\eta)$, where $\veps_n\leq \ti\veps_n \to 0$ is another sequence as usual, hence contradicting the minimality of $K$. We already know that $\bfd(c_n)\leq \veps_n$ and $\bfd(d_n)\geq \eta$. Thus it suffices to find $r_n>0$ such that $\lim_{n\to \infty}\sup_{t\in [c_n,d_n]} \bfd_{K-1}(t;r_n)=0$ as $n\to \infty.$

First, using~\eqref{eq:lambda'-bound}
\EQ{ \label{eq:lambda'-dn} 
\abs{\lambda_j(t)^2 - \lambda_j(c_n)^2} \le C_0 (t- c_n)
}
for all $t \in [c_n, d_n]$ and all $j=1,\cdots,N.$ Hence, using the contradiction assumption~\eqref{eq:short-time}  we can ensure that for large enough $n$,  
\EQ{
\frac{3}{4} \le  \frac{ \lam_j(t)}{\lam_j(c_n)}   \le \frac{5}{4} 
}
for all $j = K, \dots, N$ and all $t \in [c_n, d_n]$.  Since $\bfd(c_n) \to 0$ we have
\EQ{ \label{eq:lamK0} 
\lim_{n \to \infty} \sup_{t \in [c_n, d_n]} \sum_{j=K}^{N} \Big( \frac{ \lam_j(t)}{\lam_{j+1}(t)} \Big)^{\frac{D-2}{2}}  = 0
} 
and furthermore there exists a sequence $(r_n)$ such that   
\EQ{ \label{eq:rn-choice} 
\lam_{K-1}(c_n) + (d_n - c_n)^{\frac{1}{2}} \ll r_n \ll \lam_K(c_n) \mand \lim_{n \to \infty} \Tilde{E}( u(c_n); \frac{1}{8} r_n, 8 r_n)  = 0.
}
By the definition of the scale $\nu(t)$ in Lemma \ref{lem:basic-modulation}, we know that $u(t)$ is $o_n(1)$ close to $(N-K)$-exterior bubbles in the region $[\nu(t),\infty)$ as $n\to \infty.$ 

Now, using $\sup_{t\in [c_n,d_n]}\nu(t)|\nu'(t)|=o_n(1)$ as $n\to \infty$ and the fact that $\lam_K(c_n) \ll \nu(c_n)$, which follows from \eqref{eqn:tilde-nu-lam_j},  we get 
\EQ{\label{eqn:nu-does-not-move}
\sup_{t\in [c_n,d_n]} \left|\frac{\nu(t)}{\nu(c_n)}-1\right| = o_n(1)
}
as $n\to \infty.$ Combining this with the exterior energy control \eqref{eqn:exterior-bubble} yields
\EQ{
\sup_{t\in [c_n,d_n]} &\left(\|u(t)-u^*-\calW({\iota}_{K+1},\ldots,{\iota}_{N}, \lam_{K+1}(t),\ldots,\lam_{N}(t))\|_{\calE(2\nu(c_n),\infty)}^2 + \left(\frac{\nu(t)}{\lam_{K+1}}\right)^{\frac{D-2}{2}} \right. \\
&\quad \left. + \sum_{j=K+1}^N \left(\frac{\lam_j(t)}{\lam_{j+1}(t)}\right)^{\frac{D-2}{2}}\right) = o_n(1)
}
as $n\to \infty$. Note that we also have $\lam_K(t) \ll \nu(t)$ and by definition $\nu(t)\ll \lam_{K+1}(t)$ for all $t\in [c_n,d_n].$ On the other hand, using \eqref{eq:d-g-lam} with $\bfd(c_n) =o_n(1)$ as $n\to \infty$ and Lemma \ref{lem:basic-modulation} we have
\EQ{
\|u(c_n)-\iota_{n,K}W_{\lam_K(c_n)}\|_{\calE(\frac{r_n}{2},4\nu(c_n))} = o_n(1)
}
as $n\to \infty.$ Applying Lemma \ref{lem:un-seq} with $t_n=d_n-c_n$ we get that
\EQ{
\sup_{t\in [c_n,d_n]}\|u(t)-\iota_{n,K}W_{\lam_K(c_n)}\|_{\calE(r_n,2\nu(c_n))} = o_n(1)
}
as $n\to \infty.$ Combining the previous observations, we get 
\EQ{
\sup_{t\in [c_n,d_n]}&\left(\|u(t)-u^*-\iota_{n,K}W_{\lam_K(c_n)}-\calW({\iota}_{K+1},\ldots,{\iota}_{N}, \lam_{K+1}(t),\ldots,\lam_{N}(t))\|_{\calE(r_n,\infty)}^2 +\left(\frac{r_n}{\lam_K(c_n)}\right)^{\frac{D-2}{2}} \right. \\
&\quad \left. + \left(\frac{\lam_K(c_n)}{\lam_{K+1}(t)}\right)^{\frac{D-2}{2}}+ \sum_{j=K+1}^N \left(\frac{\lam_j(t)}{\lam_{j+1}(t)}\right)^{\frac{D-2}{2}}\right) = o_n(1)
}
as $n\to \infty.$ Thus, 
\EQ{
\sup_{t\in [c_n,d_n]}\bfd_{K-1}(t;r_n) = o_n(1)
}
as $n\to \infty$ which contradicts the minimality of $K.$ 

\end{proof} 

\begin{lem}\label{lem:cndn} For any $\eta>0$ as in Lemma~\ref{lem:collision-duration}, $\veps_n \to 0$ be some sequence, and let $[a_n, b_n] \in \calC_K(\veps_n, \eta)$. Then, there exist $\veps \in (0, \eta)$,  $n_0 \in \N$,  and $[c_n,  d_n] \subset (a_n, b_n)$ such that for all $n \ge n_0$, we have 
\EQ{ \label{eq:d>eps} 
\bfd(t) \ge \veps, \quad \forall \, \, t \in [c_n, d_n], 
}
\EQ{\label{eq:dn-cn} 
d_n - c_n = \frac{1}{n} \lam_K(c_n)^2, 
}
and 
\EQ{ \label{eq:lamKcn} 
\frac{1}{2} \lam_K(c_n) \le \lam_K(t) \le 2\lam_K(c_n) \quad \forall\, \, t \in [c_n, d_n]. 
} 
\end{lem} 
\begin{proof} 
We follow the argument outlined in the proof of Lemma 6.2 in \cite{lawrie-harmonic-map}. Let $\veps>0$ so that Lemma~\ref{lem:collision-duration} holds and define $c_n:= \sup\{t \in [a_n, b_n] \mid \bfd(t) \le \veps\}$. Then $\bfd(c_n) = \veps$ and by Lemma~\ref{lem:collision-duration} we have 
\EQ{
b_n - c_n \ge C_u^{-1} \lam_K(c_n). 
}
Set $d_n:= c_n + \frac{1}{n} \lam_K(c_n)^2$ and for $n$ sufficiently large we have $d_n < b_n$. Then by~\eqref{eq:lambda'-bound} we have, 
\EQ{
\Big| \frac{ \lam_K(t)^2}{\lam_K(c_n)^2} - 1 \Big| \lesssim \frac{ d_n- c_n}{\lam_K(c_n)}  = \frac{1}{n}
}
which implies ~\eqref{eq:lamKcn}.
\end{proof} 
\begin{proof}[Proof of Theorem~\ref{thm:main}] We proceed by a contradiction argument. Suppose that Theorem~\ref{thm:main} is false. Then let $[a_n, b_n] \in \calC_K(\veps_n, \eta)$ be a sequence of disjoint collision intervals as in Lemma ~\ref{lem:basic-modulation}, and let $\eta>0$ be sufficiently small such that Lemma~\ref{lem:collision-duration} and Lemma~\ref{lem:delta-to-d} hold. Let $\veps>0$, $n_0$, and $[c_n, d_n]$ be as in Lemma~\ref{lem:cndn}. 

We first show that there exists a constant $c_0>0$ such that for every $n \ge n_0$, 
\EQ{ \label{eq:tension-lower} 
\inf_{t \in [c_n, d_n]} \lam_K(t)^2 \| \p_t u(t) \|_{L^2}^2 \ge c_0. 
}
If not, we can find a sequence $s_n \in [c_n, d_n]$ (after possibly passing to a subsequence) such that
\EQ{
\lim_{n \to \infty} \lam_K(s_n) \| \p_t u(s_n) \|_{L^2}  = 0.
}
However, then Lemma~\ref{lem:compactness} yields a sequence $r_n \to \infty$ such that, after passing to a further subsequence,  
\EQ{
\lim_{n \to \infty}  \delta_{r_n \lam_K(s_n)} ( u(s_n)) = 0.
}
Now denote $\mu_n=r_n \lam_K(s_n)$. Then, since $\nu(s_n)=\nu_n \lam_{K+1}(s_n)$ if $\mu_n \leq \nu(s_n)$, by Lemma \ref{lem:basic-modulation} $\|u(s_n)\|_{\calE(\mu_n,\nu(s_n))}\leq \eta_0$ since
\EQ{
\|u(s_n)\|_{\calE(\mu_n,\nu(s_n))} \leq  \|g(s_n)\|_{\calE(\mu_n,\nu(s_n))} + \|\calW(\vec{\io},\vec\lam)\|_{\calE(\mu_n,\nu(s_n))} \leq C_0\eta \leq \eta_0,
}
since when $T_+<\infty$, $u^*$ is negligible in the region $(\mu_n,\nu(s_n))$ and we can assume that $\eta>0$ is small enough. Thus, applying Lemma~\ref{lem:delta-to-d} implies that
\EQ{
\lim_{n \to \infty} \bfd(s_n) = 0
}
contradicting~\eqref{eq:d>eps}. As a consequence~\eqref{eq:tension-lower} holds. Therefore, using~\eqref{eq:tension-lower},~\eqref{eq:lamKcn}, and~\eqref{eq:dn-cn} we have 
\EQ{
\sum_{n \ge n_0} \int_{c_n}^{d_n} \| \p_t u(t) \|_{L^2}^2 \, \ud t  \ge \frac{c_0}{4} \sum_{n \ge n_0} \int_{c_n}^{d_n} \lam_K(c_n)^{-2} \,  \ud t \ge \frac{c_0}{4} \sum_{n \ge n_0} n^{-1} = \infty.
}
On the other hand,  by~\eqref{eq:tension-L2} and the fact that the $[c_n, d_n]$ are disjoint we have
\EQ{
\sum_{n \ge n_0} \int_{c_n}^{d_n} \| \p_t u(t) \|_{L^2}^2 \, \ud t  \le \int_0^{T_+} \| \p_t u(t) \|_{L^2}^2 \, \ud t < \infty, 
}
which is a contradiction. 
\end{proof}

\bibliographystyle{alpha}
\bibliography{refs}
\end{document}